\documentclass[11pt,leqno,a4paper]{amsproc}

\usepackage[]{hyperref}
\hypersetup{
    colorlinks=true,       
    linkcolor=red,          
    citecolor=blue,        
    filecolor=magenta,      
    urlcolor=cyan           
}

\usepackage{amsmath}
\usepackage{amsfonts,amssymb}
\usepackage{enumerate}
\usepackage{mathrsfs}
\usepackage{amsthm}
\usepackage{booktabs}
\usepackage{multirow}
\usepackage{colortbl}

\theoremstyle{plain}
\newtheorem{theorem}{Theorem}[section]
\newtheorem*{theo*}{Theorem}
\newtheorem{proposition}[theorem]{Proposition}
\newtheorem{lemma}[theorem]{Lemma}

\newtheorem{definition}[theorem]{Definition}
\theoremstyle{definition}
\newtheorem{remark}[theorem]{Remark}
\newtheorem{example}[theorem]{Example}
\DeclareMathOperator{\cnx}{div}
\DeclareMathOperator{\cn}{div}

\DeclareMathOperator{\diff}{d}
\DeclareSymbolFont{pletters}{OT1}{cmr}{m}{sl}
\DeclareMathSymbol{s}{\mathalpha}{pletters}{`s}
\makeatletter
\protected\def\ccell#1#{%
  \kern-\fboxsep
  \@ccell{#1}%
}
\def\@ccell#1#2#3{%
  \colorbox#1{#2}{#3}%
  \kern-\fboxsep
}
\makeatother

\def\ah{\arrowvert_{y=h}}

\def\ba{\begin{align}}
\def\bad{\begin{aligned}}
\def\be{\begin{equation}}
\def\ea{\end{align}}
\def\ead{\end{aligned}}
\def\ee{\end{equation}}
\def\e{\eqref}

\def\Hm{\mathcal{H}^{d}}

\def\dHm{\diff \! \Hm}
\def\dt{\diff \! t}

\def\dx{\diff \! x}

\def\dydx{\diff \! y \diff \! x}

\def\fract{\frac{\diff}{\dt}}
\def\fractt{\frac{\diff^2}{\dt^2}}

\def\cnxy{\cn_{x,y}}

\def\defn{\mathrel{:=}}

\def\eps{\varepsilon}

\def\la{\left\vert}
\def\lA{\left\Vert}
\def\le{\leq}

\def\mez{\frac{1}{2}}
\def\partialx{\nabla}

\def\ra{\right\vert}
\def\rA{\right\Vert}

\def\uq{\frac{1}{4}}

\def\xN{\mathbf{N}}
\def\xR{\mathbf{R}}

\def\xT{\mathbf{T}}

\numberwithin{equation}{section}
\date{}
\title{Functional inequalities and strong Lyapunov functionals for free surface flows in fluid dynamics}
\author{Thomas Alazard and Didier Bresch}

\begin{document}

\setlength{\baselineskip}{5mm}

\begin{abstract}
This paper is motivated by the study of Lyapunov functionals for four equations describing free surface flows in fluid dynamics: 
the Hele-Shaw and Mullins-Sekerka equations
together with their lubrication approximations,            
the Boussinesq and thin-film equations. We identify
new Lyapunov functionals, including some which 
decay in a convex manner  (these are called strong Lyapunov functionals). 
For the Hele-Shaw equation and the Mullins-Sekerka equation, 
 we prove that the $L^2$-norm of the free surface elevation and 
the area of the free surface are Lyapunov functionals, together with parallel 
results for the thin-film and Boussinesq equations. 
The proofs combine exact identities for the dissipation rates with 
functional inequalities. For the thin-film and Boussinesq equations, 
we introduce a Sobolev inequality of independent interest which revisits some 
known results and exhibits strong Lyapunov functionals. For the 
Hele-Shaw and Mullins-Sekerka equations, we introduce a functional 
which controls the $L^2$-norm of three-half spatial derivative. 
Under a mild smallness assumption on the initial data, we show that the 
latter quantity is also a Lyapunov functional for the Hele-Shaw equation, implying 
that the area functional is a strong Lyapunov functional. 
Precise lower bounds for the dissipation rates are established, showing that these 
Lyapunov functionals are in fact entropies.
Other quantities are also studied such as Lebesgue norms or the Boltzmann's entropy.
\end{abstract}

\maketitle

\section{Introduction}
\subsection*{The equations}
Consider a time-dependent surface $\Sigma$ given as 
the graph of some function $h$, so that at time $t\ge 0$,
$$
\Sigma(t)=\{ (x,y) \in \xT^{d}\times \xR\,;\, y = h(t,x)\},
$$
where $\xT^{d}$ denotes a $d$-dimensional torus. 
We are interested by several free boundary problems described by nonlinear parabolic equations. 
A free boundary problem is described by an evolution equation which expresses  
the velocity of $\Sigma$ at each point 
in terms of some nonlinear expressions depending on $h$. 
The most popular example is the \textbf{mean-curvature} equation, which stipulates that 
the normal component of the velocity of $\Sigma$ is 
equal to the mean curvature at each point. It follows that:
\be\label{defi:kappa}
\partial_t h+\sqrt{1+|\nabla h|^2}\kappa=0\quad\text{where}\quad
\kappa=-\cnx \left(\frac{\nabla h}{\sqrt{1+|\nabla h|^2}}\right).
\ee
The previous equation plays a fundamental role in differential geometry. 
Many other free boundary problems appear in fluid dynamics. 
Among these, we are chiefly concerned by the equations modeling 
the dynamics of a free surface transported by the flow of an incompressible 
fluid evolving according to Darcy's law. We begin with the Hele-Shaw equations with or without 
surface tension. One formulation of this problem reads (see Appendix~\ref{appendix:HS}):
\be\label{HS}
\partial_{t}h+G(h)(gh+\mu \kappa)=0,
\ee
where $\kappa$ is as in~\e{defi:kappa}, $g$ and $\mu$ are real numbers in $[0,1]$ and 
$G(h)$ is the (normalized) Dirichlet-to-Neumann operator, defined as follows: For any functions $h=h(x)$ and $\psi=\psi(x)$,
$$
G(h)\psi (x)=\sqrt{1+|\nabla h|^2}\partial_n\mathcal{H}(\psi)
\big\arrowvert_{y=h(x)},
$$
where $\nabla=\nabla_x$, $\partial_n=n\cdot\nabla$ and $n$ is the outward unit normal to $\Sigma$ given by
$$
n=\frac{1}{\sqrt{1+|\nabla h|^2}}\begin{pmatrix} -\nabla h\\ 1\end{pmatrix},
$$
and $\mathcal{H}(\psi)$ is the harmonic extension of~$\psi$ in the fluid domain, solution to 
\be\label{defi:varphiintro}
\left\{
\begin{aligned}
&\Delta_{x,y}\mathcal{H}(\psi)=0\quad \text{in }\Omega\defn\{(x,y)\in \xT^{d}\times \xR\,:\, y<h(x)\},\\
&\mathcal{H}(\psi)\arrowvert_{y=h}=\psi.
\end{aligned}
\right.
\ee
Hereafter, given a function $f=f(x,y)$, we use $f\arrowvert_{y=h}$ as a short notation for the function $x\mapsto f(x,h(x))$. 

When $g=1$ and $\mu=0$, the equation~\e{HS} is called the Hele-Shaw equation without surface tension. 
Hereafter, we will refer to this equation simply as the \textbf{Hele-Shaw} equation. 
If $g=0$ and $\mu=1$, the equation is known as the Hele-Shaw equation with surface tension, 
also known as the \textbf{Mullins-Sekerka} equation. 
Let us record the terminology:
\begin{alignat}{2}
&\partial_{t}h+G(h)h=0\qquad &&(\text{Hele-Shaw}),\label{HSi}\\
&\partial_{t}h+ G(h)\kappa=0\qquad &&(\text{Mullins-Sekerka})\label{MSi}.
\end{alignat}
We are also interested by two equations which 
describe asymptotic regime in the \textbf{thin-film} approximation. They are 
\begin{align}
&\partial_t h-\cnx(h\nabla h)=0 \qquad&&(\text{Boussinesq}),\label{Bou}\\
&\partial_t h+\cnx(h\nabla \Delta h)=0\qquad &&(\text{thin-film}).\label{ThFi}
\end{align}
Equation \e{Bou} was derived from~\e{HSi} by Boussinesq~\cite{Boussinesq-1904} 
to study groundwater infiltration. 
Equation \e{ThFi} was derived from~\e{MSi} by Constantin, Dupont, Goldstein, Kadanoff, Shelley and Zhou in~\cite{Constantin1993droplet} 
as a lubrication approximation model of the interface between two immiscible fluids in a Hele-Shaw cell.

\subsection{Lyapunov functionals and entropies}
Our main goal is to find some 
monotonicity properties 
for the previous free boundary flows, in a unified way. 
Before going any further, let us fix the terminology used in this paper. 
\begin{definition}\label{Defi:1.1}
(a) Consider one of the evolution equation stated above and a function
$$
I: C^\infty(\xT^{d})\to [0,+\infty).
$$
We say that $I$ is a \textbf{Lyapunov functional} if the following property holds: 
for any smooth solution $h$ in $C^\infty([0,T]\times \xT^{d})$ for some $T>0$, we have
$$
\forall t\in [0,T],\qquad \fract I(h(t))\le 0. 
$$
The quantity $-\fract I(h)$ is called the \textbf{dissipation rate} of the functional $I(h)$.

(b) We say that a Lyapunov functional $I$ is an \textbf{entropy} if the dissipation rate satisfies, for some $C>0$,
$$
-\fract I(h(t))\ge C I(h(t)).
$$

(c) Eventually, we say that $I$ is a \textbf{strong Lyapunov functional} if 
$$
\fract I(h(t))\le 0\quad\text{and}\quad\fractt I(h(t))\ge 0.
$$
This means that $t\mapsto I(h(t))$ decays in a convex manner. 
\end{definition}
\begin{remark}
$(i)$ The Cauchy problems for the previous free boundary equations have been studied by  different techniques, 
for weak solutions, 
viscosity solutions or also classical solutions. We refer the reader to \cite{A-Lazar,AMS,ChangLaraGuillenSchwab,Chen-ARMA-1993,ChengCoutandShkoller-2014,Cheng-Belinchon-Shkoller-AdvMath,ChoiJerisonKim,CCG-Annals,Escher-Simonett-ADE-1997,FlynnNguyen2020,GG-JPS-AdvMaths-2019,Gunther-Prokert-SIAM-2006,
Hadzic-Shkoller-CPAM2015,Kim-ARMA2003,Knupfer-Masmoudi-ARMA-2015,NPausader,Pruss-Simonett-book}. 
Thanks to the parabolic smoothing effect, classical solutions are smooth for positive times 
(the elevation $h$ belongs to $C^\infty((0,T]\times \xT^{d})$). 
This is why we consider functionals $I$ defined only on smooth functions $C^\infty(\xT^{d})$. 

$(ii)$ Assume that $I$ is an entropy for an evolution equation and consider a global in time 
solution of the latter problem. 
Then the function 
$t\mapsto I(h(t))$ decays exponentially fast. In the literature, there are more general 
definition of entropies for various 
evolution equations. The common idea is that entropy dissipation methods allow 
to study the large time behavior or to prove functional inequalities 
(see~\cite{Bertozzi-NoticesAMS-1998,Carillo-Jungel-Markovich-Toscani-Unterreiter,Arnold-et-al-2004,Evans-2004,Villani-Oldandnew,Bolley-Gentil-JMPA-2010,Dolbeault-Toscani-AIHPNL-2013,Bodineau-Lebowitz-Mouhot-Villani,Zugmeyer-arxiv2020,Jungel-book-entropy}).

$(iii)$ To say that $I(h)$ is a strong Lyapunov functional 
is equivalent to say that the dissipation rate $-\fract I(h)$ is also a Lyapunov functional. 
This notion was introduced in~\cite{Aconvexity} 
as a tool to find Lyapunov functionals which control higher order Sobolev norms.  
Indeed, in general, 
the dissipation rate is expected to be a higher order energy 
because of the smoothing effect of a parabolic equation. 
Notice that the idea to compute the second-order derivative in time is related to the 
celebrated work of Bakry and Emery~\cite{BakryEmmery-1985}.
\end{remark}

\subsection{Examples}

Since we consider different equations, for the reader's convenience, 
we begin by discussing some examples which are well-known in certain 
communities.

\begin{example}\label{Example:heateq}
Consider the heat equation
$\partial_t h-\Delta h=0$. The energy identity
$$
\mez\fract \int_{\xT^{d}}h^2\dx +\int_{\xT^{d}}\la \nabla h\ra^2\dx=0,
$$
implies that the square of the $L^2$-norm is a Lyapunov functional. 
It is in addition a strong Lyapunov 
functional since, by differentiating the equation, the quantity $\int_{\xT^{d}}\la \nabla h\ra^2\dx$ 
is also a Lyapunov functional. Furthermore, if one assumes that the mean value of $h(0,\cdot)$ vanishes, then 
the Poincar\'e's inequality implies that the square of the $L^2$-norm is an entropy. 
Now let us discuss another important property, which holds for positive solutions. 
Assume that $h(t,x)\ge 1$ and introduce the Boltzmann's entropy, defined by
$$
H(h)=\int_{\xT^{d}}h \log h \dx.
$$
Then $H(h)$ is a strong Lyapunov functional. 
This classical result (see Evans~\cite{Evans-BAMS-2004}) 
follows directly from the pointwise identities
\begin{align*}
&(\partial_t-\Delta)(h\log h)=-\frac{\la \nabla h\ra^2}{h},\\
&(\partial_t -\Delta )\frac{\la \nabla h\ra^2}{h}=-2h\la \frac{\nabla^2 h}{h}
-\frac{\nabla h \otimes \nabla h}{h^2}\ra^2.
\end{align*}
We will prove that the Boltzmann's entropy is also a strong Lyapunov functional for the Boussinesq equation~\e{Bou}, by using 
a functional inequality which controls the $L^2$-norm of 
$\la \nabla h\ra^2/h$. Recall that the $L^1$-norm of $\la \nabla h\ra^2/h$, called the Fisher's information, 
plays a key role in entropy methods and 
information theory (see~ Villani's lecture notes~\cite{Villani-Lecturenotes2008} and his book~\cite[Chapters 20, 21, 22]{Villani-Oldandnew}).
\end{example}
For later references and comparisons, we discuss some examples of Lyapunov functionals 
for the nonlinear equations mentioned above.
\begin{example}[Mean-curvature equation]\label{example:MCF}
Consider the mean curvature equation $\partial_t h+\sqrt{1+|\nabla h|^2}\kappa=0$. 
If $h$ is a smooth solution, then
\be\label{MCF:n0}
\fract \Hm(\Sigma)\le 0 \quad\text{where}\quad
\Hm(\Sigma)=\int_{\xT^{d}}\sqrt{1+|\nabla h|^2}\dx.
\ee
This is proved by an integration by parts argument:
\begin{align*}
\fract \Hm(\Sigma)&=\int_{\xT^{d}}\nabla_x (\partial_th) \cdot \frac{\nabla_x h}{\sqrt{1+|\nabla h|^2}}\dx
=\int_{\xT^{d}} (\partial_t h)\kappa\dx\\
&=-\int_{\xT^{d}}\sqrt{1+|\nabla h|^2}\kappa^2\dx\le 0.
\end{align*}
In fact, the mean-curvature equation is a gradient flow for $\Hm(\Sigma)$, see~\cite{CMWP-BAMS-2015}. 
When the space dimension $d$ is equal to $1$, 
we claim that the following quantities are also Lyapunov functionals:
$$
\int_\xT h^2\dx,\quad \int_\xT (\partial_x h)^2\dx,\quad 
\int_\xT (\partial_t h)^2\dx,\quad \int_\xT (1+(\partial_xh)^2)\kappa^2\dx. 
$$
To our knowledge, these results are new and we prove 
this claim in Appendix~\ref{Appendix:MCF}. 
We will also prove that $\int_\xT h^2\dx$ is a strong Lyapunov functional.
\end{example}
\begin{example}[Hele-Shaw equation]\label{example:Hele-Shaw}
Consider the equation $\partial_{t}h+G(h)h=0$. 
Recall that $G(h)$ is a non-negative operator. Indeed, 
denoting by $\varphi=\mathcal{H}(\psi)$ the harmonic extension of $\psi$ given by~\e{defi:varphiintro}, 
it follows from Stokes' theorem that
\be\label{positivityDNintro}
\int_{\xT^{d}} \psi G(h)\psi\dx=\int_{\partial\Omega}\varphi \partial_n \varphi\diff\Hm=
\iint_{\Omega}\la\nabla_{x,y}\varphi\ra^2\dydx\ge 0.
\ee
Consequently, if $h$ is a smooth-solution to $\partial_{t}h+G(h)h=0$, then 
$$
\mez\fract \int_{\xT^{d}}h^2\dx =-\int_{\xT^{d}} hG(h)h\dx\le 0.
$$
This shows that $\int_{\xT^{d}} h^2\dx$ is a Lyapunov functional. 
In \cite{AMS}, it is proved that in fact $\int_{\xT^{d}} h^2\dx$ 
is a strong Lyapunov functional and also an entropy. 
This result is generalized in \cite{Aconvexity} to functionals of the form 
$\int_{\xT^{d}} \Phi(h)\dx$ where $\Phi$ is a convex function whose 
derivative is also convex. 
\end{example}
\begin{example}[Mullins-Sekerka]\label{example:Mullins-Sekerka}
Assume that $h$ solves $\partial_{t}h+G(h)\kappa=0$ and denote by $\Hm(\Sigma)$ the area functional (see~\e{MCF:n0}). Then~\e{positivityDNintro} implies that
$$
\fract \Hm(\Sigma)=\int_{\xT^{d}} (\partial_t h)\kappa\dx=-\int_{\xT^{d}}\kappa G(h)\kappa\dx\le 0,
$$
so $\Hm(\Sigma)$ is a Lyapunov functional. 
In fact, the Mullins-Sekerka equation 
is a gradient flow for $\Hm(\Sigma)$, 
see~\cite{Almgren-Physics-1996,Giacomelli-Otto-CVPDE-2001}. 
\end{example}
\begin{example}[Thin-film equation]\label{exampleTF}
The study of entropies plays a key role in the study of the thin-film equation (and its variant) 
since the works of Bernis and Friedman~\cite{Bernis-Friedman-JDE} 
and Bertozzi and Pugh~\cite{Bertozzi-Pugh-1996}. The simplest observation is that, 
if $h$ is a non-negative solution to 
$\partial_th+\partial_x(h\partial_x^3 h)=0$, then
$$
\fract \int_\xT h^2\dx\le 0, \qquad \fract \int_\xT (\partial_x h)^2\dx\le 0.
$$
(This can be verified by elementary integrations by parts.) To give an example of 
hidden Lyapunov functionals, consider, for $p\ge 0$ and a function $h> 0$, the functionals
$$
H_p(h)=\int_\xT \frac{h_x^2}{h^p}\dx.
$$
Laugesen discovered~(\cite{Laugesen-CPAA}) that, for $0\le p\le 1/2$, $H_p(h)$ 
is a Lyapunov functional. 
This result was complemented by 
Carlen and Ulusoy~(\cite{Carlen-Ulusoy-CMS}) who showed 
that $H_p(f)$ is an entropy when $0< p<(9 + 4\sqrt{15})/53$. 
We also refer to \cite{BerettaBDP-ARMA-1995,DPGG-Siam-1998,BDPGG-ADE-1998,JungelMatthes-Nonlinearity-2006} 
for the study of entropies of the form $\int h^p\dx$ with $1/2\le p\le 2$.
\end{example}

\subsection{Main results and plan of the paper} 
We are now ready to introduce our main new results. 
To highlight the links between them, we 
begin by gathering in the following table the list of all the Lyapunov functionals 
that will be considered. This table includes known results, some of which have already been 
discussed and others will be reviewed later. 
Precise statements are given in the next section.

\begin{tabular}{@{}llllr@{}}
  \toprule
    Equations &  \multicolumn{3}{c} { \textbf{Lyapunov functionals} \hfill See} &  Properties  \\[1ex]
      \toprule
  \textbf{Heat} & {$\int h^2$ } &$(*)$ & Ex.~\ref{Example:heateq} &(S) \\[0.5ex]
    \textbf{equation} & $\int h \log h$   & $(*)$& Ex.~\ref{Example:heateq} &(S), (GF) \\[0.5ex]
   \midrule
  \textbf{Mean} & $\int \sqrt{1+|\nabla h|^2}=\Hm(\Sigma)$ \hfill  &$(*)$& Ex.~\ref{example:MCF} & (GF)\\[0.5ex]
    \textbf{curvature}& \cellcolor[gray]{0.95}{$\int \la\nabla h\ra^2$}& &Prop.~\ref{Prop:C1nabla} & \\[0.5ex]
     & \cellcolor[gray]{0.95}{$\int h^2$} & $(d=1)$&Prop.~\ref{Prop:C1} & \ccell[gray]{0.95}{(S)}  \\[0.5ex]
    
    & \cellcolor[gray]{0.95}{$\int (\partial_th)^2$} &$(d=1)$ &Prop.~\ref{Prop:C1} &\\[0.5ex]
    & \cellcolor[gray]{0.95}{$\int (1+(\partial_xh)^2)\kappa^2$} \hfill&$(d=1)$ &Prop.~\ref{Prop:C1}& \\[0.5ex]
    & \cellcolor[gray]{0.95}{$\int (\partial_xh)\arctan (\partial_xh)$} & $(d=1)$ &Prop.~\ref{Prop:C1}& \\[0.5ex]
    \midrule
    \textbf{Hele-Shaw} & $\int \Phi(h)$, $\Phi''\ge 0$\hfill &$(*)$& Ex.~\ref{example:Hele-Shaw} &  \\[0.5ex]
    & $\int \Phi(h)$ , $\Phi'',\Phi'''\ge 0$  &$(*)$& Ex.~\ref{example:Hele-Shaw} & (S) \\[0.5ex]
    & $\int h G(h)h$  &$(*)$& \S\ref{S213} &  \\[0.5ex]
    & \cellcolor[gray]{0.95}{$\int \sqrt{1+|\nabla h|^2}$}  && Th.~\ref{T1} &\ccell[gray]{0.95}{(S)}\\[0.5ex]
    & \cellcolor[gray]{0.95}{$\int \kappa G(h)h$} && Th.~\ref{Theorem:J(h)decays} &\\[0.5ex]
    \midrule
    \textbf{Mullins-} & $\int \sqrt{1+|\nabla h|^2}$ \hfill  & $(*)$ &Ex.\ref{example:Mullins-Sekerka} &  (GF) \\[0.5ex]
    \textbf{Sekerka} & \cellcolor[gray]{0.95}{$\int h^2$} & &Th.~\ref{T1} &  \\[1ex]
    \midrule
   \textbf{Thin-film} & $\int \la\nabla h\ra^2$ & $(*)$& Prop.~\ref{prop:lubrik1n} &  \\[0.5ex]
        &$\int h^{-p}h_x^2\qquad 0\le p\le 1/2$ &$(*)$& Ex.~\ref{exampleTF}& \\[0.5ex]
    & \cellcolor[gray]{0.95}{$\int h^{m}\qquad\quad \mez \le  m\le 2$} & (**) &Prop.~\ref{positivity} &\\[0.5ex]
    & $\int h\log h$ & &Prop.~\ref{positivity}& \\[0.5ex]\midrule
  \textbf{Boussinesq} & {$\int h ^2$}& & Th.~\ref{Theo2bis} & \ccell[gray]{0.95}{(S)} \\[0.5ex]
    & {$\int h\log h$}  && Th.~\ref{Theo2bis} &  \ccell[gray]{0.95}{(S)}\\[0.5ex]
    & $\int h^{m+1}$ & (*) & Prop.~\ref{convexporoust} &\\[0.5ex]
    & {$\int h^2\la\nabla h\ra^2$}&($*$) &\S\ref{S:Boussinesq}& \\[0.5ex]
    & \cellcolor[gray]{0.95}{$\int h^m\la\nabla h\ra^2, ~ 0\le m\le \frac{1+\sqrt{7}}{2}$} &$(**)$ &Prop.~\ref{convexporous}& \\[0.5ex]
    & \cellcolor[gray]{0.95}{$\int (\partial_xh)\arctan (\partial_xh)$}  &$(d=1)$ &Prop.~\ref{prop:C2Boussinesq}& \\[0.5ex]
    \bottomrule
    \textbf{Legend:} & \multicolumn{4}{l} {The gray boxes point to the new results}     \\
    & \multicolumn{4}{l} { \small{$(*)$: already known}}    \\
    & \multicolumn{4}{l} { \small{$(**)$: improves previous exponents or simplifies the proof}}    \\
    & \multicolumn{4}{l} { \small{$(d=1)$: only in dimension one}}    \\
    & \multicolumn{4}{l} { \small{(S): is a strong Lyapunov functional}}    \\
    & \multicolumn{4}{l} { \small{(GF): is derived from a Gradient Flow structure.}}  \\
 \bottomrule
\end{tabular}

To conclude this introduction, let us mention that in addition to Lyapunov functionals, 
maximum principles also play a key role 
in the study of these parabolic equations. One can think of the 
maximum principles for the mean-curvature equation obtained 
by Huisken~\cite{Huisken-JDE-1984} and Ecker and Huisken~(see \cite{Ecker-Huisken-Annals,Ecker-Regularity-Theory}), 
used to obtain a very sharp global existence result 
of smooth solutions. Many maximum principles exist also 
for the Hele-Shaw equations (see~\cite{Kim-ARMA2003,ChangLaraGuillenSchwab}). In particular, we will use the 
maximum principle for 
space-time derivatives proved in~\cite{AMS}. Sea also~\cite{ConstantinVicol-GAFA2012} for related models. 
For the thin-film equations of the form $\partial_th+\partial_x(f(h)\partial_x^3 h)=0$ with $f(h)=h^m$ and an exponent $m\ge 3.5$, 
in one space dimension, 
if the initial data $h_0$ is positive, then the solution $h(x,t)$ 
is guaranteed to stay positive 
(see~\cite{Bernis-Friedman-JDE,Bertozzi-et-al-1994} 
and~\cite{DPGG-Siam-1998,BDPGG-ADE-1998,ZhornitskayaBertozzi-2000,Bresch2018bd}).

\section*{Acknowledgements} 
The authors acknowledge the support of the SingFlows project, 
grant ANR-18-CE40-0027 of the French National Research
Agency (ANR).

\section{Statements of the main results}

Our main goal is to study the decay 
properties of several natural coercive quantities for the Hele-Shaw, 
Mullins-Sekerka, Boussinesq and thin-film equations, 
in a unified way. 

\subsection{Entropies for the Hele-Shaw and Mullins-Sekerka equations}
The first two coercive quantities we want to study are 
the $L^2$-norm and the area functional (that is the $d$-dimensional surface measure):
\be\label{L2Hm}
\left(\int_{\xT^{d}}h(t,x)^2\dx\right)^\mez,\qquad \Hm(\Sigma)=\int_{\xT^{d}}\sqrt{1+|\nabla h|^2}\dx.
\ee
Our first main result states that these are Lyapunov functionals 
for the Hele-Shaw and Mullins-Sekerka equations, 
in any dimension.

\begin{theorem}\label{T1}
Let $d\ge 1$, $(g,\mu)\in [0,+\infty)^2$ and assume that $h$ 
is a smooth solution to 
\be\label{n21}
\partial_{t}h+G(h)(gh+\mu \kappa)=0.
\ee
Then, 
\be\label{n31}
\fract \int_{\xT^{d}} h^2\dx\le 0 \quad 
\text{and}\quad \fract \Hm(\Sigma)\le 0. 
\ee
\end{theorem}
\begin{remark}
The main point is that this result holds uniformly with respect to $g$ and $\mu$. 
For comparison, let us recall some results which hold for the special cases where either $g=0$ or $\mu=0$. 

$i)$ When $g=0$, the fact that the area-functional $\Hm(\Sigma)$ decays in time 
follows from 
a well-known gradient flow structure for the Mullins-Sekerka equation. However, the 
decay of the $L^2$-norm in this case is new. 

$ii)$ When $\mu=0$, the decay of the $L^2$-norm follows from an elementary energy estimate. 
However, the proof of the decay of the area-functional $t\mapsto \Hm(\Sigma(t))$ 
requires a more subbtle argument. It is implied (but only implicitly)  
by some computations by Antontsev, Meirmanov, and Yurinsky in \cite{Antontsev-al-2004}. 
The main point is that we shall give a different approach which holds uniformly with respect to $g$ and $\mu$. 
In addition, we will obtain a precise lower bound for 
the dissipation 
rate showing that 
$\Hm(\Sigma)$ is an entropy when $\mu=0$ and not only a Lyapunov functional. 
\end{remark}
To prove these two uniform decay results, the key ingredient will be to study the following functional:
$$
J(h)\defn \int_{\xT^{d}} \kappa\,  G(h)h\dx.
$$
It appears naturally when performing energy estimates. Indeed, by multiplying the equation~\e{n21} 
with $h$ or $\kappa$ and integrating by parts, one obtains
\begin{align}
&\mez\fract\int_{\xT^{d}} h^2\dx+g \int_{\xT^{d}}hG(h)h\dx+\mu J(h)=0,\notag\\
&\fract \Hm(\Sigma)+gJ(h)+\mu\int_{\xT^{d}} \kappa G(h)\kappa\dx=0.\label{J(h)dt}
\end{align}
We will prove that $J(h)$ is non-negative. 
Since the Dirichlet-to-Neumann operator is a non-negative operator~(see \e{positivityDNintro}), 
this will be sufficient to conclude that the $L^2$-norm and the area functional $\Hm(\Sigma)$ are non-increasing along the flow.

An important fact is that $J(h)$ is a nonlinear analogue 
of the homogeneous $H^{3/2}$-norm. A first way to give this statement 
a rigorous meaning consists in noticing that $G(0) h=\la D_x\ra h=\sqrt{-\Delta_x}h$ and 
the linearized version of $\kappa$ is 
$-\Delta_x h$. Therefore, if $h=\eps\zeta$, then
$$
J(\eps \zeta)=\eps^2\int_{\xT^{d}} \big( \la D_x\ra^{3/2}\zeta\big)^2\dx+O(\eps^3).
$$
We will prove a functional inequality (see Proposition~\ref{P:Positive2} below) which shows 
that $J(h)$ controls the $L^2(\Omega)$-norm 
of the Hessian of the harmonic extension $\mathcal{H}(h)$ of~$h$, given by~\eqref{defi:varphiintro} with $\psi=h$. 
Consequently, $J(h)$ controls three half-derivative of $h$ in $L^2$ 
by means of a trace theorem.

\subsection{The area functional is a strong Lyapunov functional}\label{S213}
As seen in Example~\eqref{example:MCF}, for the mean-curvature equation in space dimension $d=1$, there exist 
Lyapunov functionals which control all the spatial derivatives of order less than $2$. 
Similarly, there are higher-order energies for the thin-film equations (see 
Theorem~\ref{Theo2}, the Laugesen's functionals introduced in Example~\ref{exampleTF} and also \cite{ConstantinElgindi}). 
On the other hand, for the Hele-Shaw and Mullins-Sekerka equations, 
it is more difficult to find higher-order energies which control some derivatives 
of the solution. This is becasue it is harder 
to differentiate these equations. 
For the Mullins-Sekerka problem, one can quote two recent papers 
by Chugreeva--Otto--Westdickenberg~\cite{ChugreevaOttoWestdickenberg2019} and 
Acerbi--Fusco--Julin--Morini~\cite{AcerbiFuscoJulinMorini2019}. 
In both papers, the authors compute the second derivative in time of some 
coercive quantities to study the long time behavior of the solutions, in perturbative regimes. 
Here, we will prove a similar result for the Hele-Shaw equation. 
However, the analysis will be entirely different. 
On the one hand, it is easier in some sense to differentiate the Hele-Shaw equation. 
On the other hand, we will be able to exploit some additional identities and inequalities which allow us to 
obtain a result under a very mild-smallness assumption.

Here, we consider the Hele-Shaw equation:
\be\label{HSJ}
\partial_t h+G(h)h=0.
\ee
It is known that Cauchy problem for the latter equation 
is well-posed on the Sobolev spaces $H^s(\xT^{d})$ provided that $s>1+d/2$, and moreover the critical 
Sobolev exponent is $1+d/2$ (see~\cite{Cheng-Belinchon-Shkoller-AdvMath,Matioc-APDE-2019,AMS,NPausader}). 
On the other hand, the natural energy estimate only controls the $L^2$-norm. 
It is thus natural to seek higher order energies, which are bounded in time and which control 
Sobolev norms $H^\mu(\xT^{d})$ of order $\mu>0$. 
It was proved in~\cite{AMS,Aconvexity} that one can control one-half derivative of $h$ by 
exploiting some convexity argument. More precisely, it is proved in the previous references that
\be\label{n120}
\fract\int_{\xT^{d}}hG(h)h\dx\le 0.
\ee
This inequality gives a control of a higher 
order Lyapunov functional of order $1/2$. Indeed,
$$
\int_{\xT^{d}}hG(h)h\dx=\iint_{\Omega}\la\nabla_{x,y}\mathcal{H}(h)\ra^2\dydx,
$$
where $\mathcal{H}(h)$ is the harmonic extension of $h$ (solution to~\e{defi:varphiintro} 
where $\psi$ is replaced by $h$). 
Hence, by using a trace theorem, 
it follows that $\int_{\xT^{d}}hG(h)h\dx$ controls the $H^{1/2}$-norm of $h$.

The search for higher-order functionals leads to interesting new difficulties. 
Our strategy here is to try to prove that the area functional is a strong Lyapunov funtional. 
This means that the function $t\mapsto\Hm(\Sigma(t))$ decays in a convex manner. 
This is equivalent to $\diff^2 \Hm(\Sigma)/\dt^2\ge 0$. Now, remembering (cf \e{J(h)dt}) 
that 
$$
\fract \Hm(\Sigma)+J(h)=0\quad\text{where}\quad J(h)=\int_{\xT^d}\kappa G(h)h\dx,
$$
the previous convexity argument suggests that $\diff J(h)/\dt\le 0$, which implies that $J(h)$ is 
a Lyapunov function. This gives us a very interesting higher-order energy since 
the functional $J(h)$ controls 
three-half spatial derivatives of $h$ (as seen above, and as will be made precise 
in Proposition~\ref{P:Positive2}). 
The next result states that the previous strategy applies under 
a very mild smallness assumption on the 
first order derivatives of the elevation $h$ at time $0$. 

\begin{theorem}\label{Theorem:J(h)decays}
Consider a smooth solution to $\partial_t h+G(h)h=0$.
There exists a universal constant $c_d$ depending only on the dimension $d$ such that, 
if initially
\be\label{esti:final6}
\sup_{\xT^d}\la \nabla h_0\ra^2 \le c_d,\qquad 
\sup_{\xT^d}\la G(h_0)h_0\ra^2 \le c_d,
\ee
then 
\be\label{n124}
\fract J(h)
 +\frac{1}{2}\int_{{\xT}^d}\frac{\big(|\nabla\nabla h|^2
 + |\nabla\partial_t h|^2\big)}{(1+|\nabla h|^2)^{3/2}}\dx\le 0.
\ee
\end{theorem}
\begin{remark}
$i)$ The constant $c_d$ is the unique solution in $[0,1/4]$ to 
$$
2c_d\left(d+\left(d+\sqrt{d}\right) c_d\right) 
+ 4 \left(c_d\left(d+ (d+1) c_d\right)\left(\frac{12}{1-2c_d}+1\right)\right)^{\mez}= \mez.
$$
$ii)$ Since
$$
\fract J(h)=- \fractt \mathcal{H}^1(\Sigma),
$$
it is equivalent to say that the area-functional $\Hm(\Sigma)$ is a strong Lyapunov functional.
\end{remark}

\subsection{Entropies for the Boussinesq and thin-film equations}
The previous theorems suggest to seek a similar uniform result for the thin-film and Boussinesq equations. 
In this direction, we will obtain various entropies and gather in the next result only the main consequences. 
\begin{theorem}\label{Theo2}
Let $d\ge 1$, $(g,\mu)\in [0,+\infty)^2$ and $h$ be a smooth solution to 
\be\label{n21B}
\partial_{t}h-\cnx \big(gh\nabla h-\mu h\nabla \Delta h\big)=0.
\ee
Then, 
\be\label{n31B}
\fract \int_{\xT^{d}} h^2\dx\le 0 \quad 
\text{and}\quad \fract \int_{\xT^{d}}\la \nabla h\ra^2\dx\le 0. 
\ee
\end{theorem}
\begin{theorem}\label{Theo2bis}
Let $d\ge 1$, and assume that $h$ is a smooth solution to
\be\label{n21bis}
\partial_{t}h-\cnx \big(h\nabla h\big)=0.
\ee
Then the square of the $L^2$-norm and the Boltzmann's entropy are strong Lyapunov functionals:
\be\label{n31.5}
\fract \int_{\xT^{d}}h^2\dx\le 0\quad \text{and}\quad\fractt \int_{\xT^{d}}h^2\dx\ge 0,
\ee
together with
\be\label{n31.5log}
\fract \int_{\xT^{d}}h\log h \dx\le 0\quad \text{and}\quad\fractt \int_{\xT^{d}}h\log h\dx\ge 0.
\ee
\end{theorem}
\begin{remark}
We will study more general Lyapunov functionals of the form $\int_{\xT^{d}} h^{m}\dx$ and 
$\int_{\xT^{d}} h^m \la\nabla h\ra^2\dx$. 
\end{remark}

When $g=0$, the first half of \e{n31B} was already obtained by several authors. 
The study of the decay of Lebesgue norms was initiated by Bernis and Friedman~\cite{Bernis-Friedman-JDE} 
and continued by 
Beretta-Bertsch-Dal Passo~\cite{BerettaBDP-ARMA-1995},
Dal Passo--Garcke--Gr\"{u}n~\cite{DPGG-Siam-1998} and more recently by 
J\"ungel and Matthes~\cite{JungelMatthes-Nonlinearity-2006}, who 
performed a systematic study of entropies for the thin-film equation, by means 
of a computer assisted proof. 
Here we will proceed differently and give a short proof, obtained 
by computations inspired by functional inequalities of Bernis~\cite{Bernis-proc-1996} and 
Dal Passo--Garcke--Gr\"{u}n~\cite{DPGG-Siam-1998}. 
Namely, we will establish a Sobolev type inequality. 
Quite surprisingly, this inequality will in turn allow us to study  the case with gravity $g>0$, which is 
in our opinion the most delicate part of the proof. 

As we will see, the crucial ingredient to prove 
Theorems~\ref{Theo2} and~\ref{Theo2bis} is given by the following functional inequality.
\begin{proposition}\label{theo:logSob}
For any $d\ge 1$ and any positive function $\theta$ in $H^2(\xT^{d})$,
\be\label{BmD}
\int_{\xT^{d}} \big|\nabla \theta^{1/2}\big|^4\dx\le \frac{9}{16}\int_{\xT^{d}} (\Delta \theta)^2 \dx.
\ee
\end{proposition}
There is a short proof which can be explained here. 
\begin{proof}
By integrating by parts, we obtain the classical observation that
\be\label{Deltanablanabla}
\begin{aligned}
\int_{\xT^{d}}(\Delta \theta)^2\dx&=\int_{\xT^d}\sum_{i,j}(\partial_i^2\theta)(\partial_j^2\theta)\dx\\
&=\int_{\xT^d}\sum_{i,j}(\partial_{ij}\theta)(\partial_{ij}\theta)\dx=\int_{\xT^{d}}\la \nabla\nabla \theta\ra^2\dx.
\end{aligned}
\ee
Now, introduce $I=16\int_{\xT^{d}} \big|\nabla \theta^{1/2}\big|^4\dx$. By an immediate computation, 
$$
I=\int_{\xT^{d}} \theta^{-2}|\nabla \theta|^4\dx
=-\int_{\xT^{d}} \big(\nabla \theta^{-1} \cdot \nabla \theta\big)  \, |\nabla\theta|^2\dx.
$$
By integrating by parts, one can rewrite $I$ under the form
$$
I=\int_{\xT^{d}} \theta^{-1} \Delta \theta |\nabla \theta|^2\dx
+2\int_{\xT^{d}} \theta^{-1}[(\nabla \theta \cdot \nabla) \nabla \theta]\cdot\nabla \theta\dx.
$$
Since $\la (\nabla \theta\cdot\nabla)\nabla \theta\ra
\le \la\nabla \theta\ra\la \nabla^2 \theta\ra$ (see~\e{n2001} for details), 
using~\eqref{Deltanablanabla} and the Cauchy-Schwarz inequality, we obtain
$$
I \le 3 \, I^{1/2} \bigg(\int_{\xT^{d}} (\Delta \theta)^2\dx\bigg)^{1/2}.
$$
Thus we conclude that
$$
I\le 9 \int_{\xT^{d}} (\Delta\theta)^2\dx,
$$ 
which is the wanted inequality.
\end{proof}
\begin{remark}
\begin{enumerate}[(i)]
\item See Proposition~\ref{P:refD.1v2} for a more general result.
\item The inequality~\e{BmD} 
is a multi-dimensional version of an inequality of Bernis which holds 
in space dimension $d=1$ (see Theorem~$1$ in~\cite{Bernis-proc-1996}). 
In this direction, notice that a 
remarkable feature of~\e{BmD} is that the constant $9/16$ 
is dimension-independent. 

\item The Bernis' inequalities in~\cite{Bernis-proc-1996} and similar ones 
(see Gr\"un~\cite{Grun-2001} and Dal Passo--Garcke--Gr\"{u}n~\cite{DPGG-Siam-1998}) 
have been used to study various problems in fluid dynamics. 
In the opinion of the authors, Proposition~\ref{theo:logSob} 
could have other applications in fluid dynamics. 
As an example, we show in Appendix~\ref{appendix:compressible} 
how to fully remove a technical obstruction 
in the construction of weak-solutions for compressible 
Navier-Stokes equations with viscosities depending on the density.
\end{enumerate}
\end{remark}

\section{Uniform Lyapunov functionals for the Hele-Shaw and Mullins-Sekerka equations}\label{S:3}

In this section, we prove Theorem~\ref{T1}.

\subsection{Maximum principles for the pressure}\label{S:pressure}
In this paragraph the time variable does not play any role and we ignore it to simplify notations.

We will need the following elementary result.
\begin{lemma}\label{Lemma:decayinfty}
Consider a smooth function $h$ in $C^\infty(\xT^d)$ and set
$$
\Omega=\{(x,y)\in\xT^{d}\times\xR\,:\,y<h(x)\}.
$$
For any $\zeta$ in $C^\infty(\xT^d)$, there is a unique function 
$\phi\in C^\infty(\overline{\Omega})$ such that 
$\nabla_{x,y}\phi\in L^2(\Omega)$, solution to 
the Dirichlet problem
\be\label{defi:varphi2-zero}
\left\{
\begin{aligned}
&\Delta_{x,y}\phi=0 \quad\text{in }\Omega,\\
&\phi(x,h(x))=\zeta(x) \text{ for all }x\in\xT^{d}.
\end{aligned}
\right.
\ee
Moreover, for any multi-index $\alpha\in\xN^d$ and any $\beta\in \xN$ with $\la\alpha\ra+\beta>0$, 
one has
\be\label{decaytozero}
\partial_x^\alpha\partial_y^\beta\phi\in L^2(\Omega)\quad \text{and}\quad
\lim_{y\to-\infty}\sup_{x\in\xT^{d}}\la \partial_x^\alpha\partial_y^\beta\phi(x,y)\ra=0.
\ee
\end{lemma}
\begin{proof}
The existence and smoothness of the solution $\phi$ is a classical elementary result. 
We prove only the property~\e{decaytozero}.

Let $y_0$ be an arbitrary real number such that $\xT^{d}\times\{y_0\}$ is located underneath the boundary $\partial\Omega=\{y=h\}$ 
and then 
set $\psi(x)=\phi(x,y_0)$. This function belongs to $C^\infty(\xT^{d})$ since $\phi$ 
belongs to $C^\infty(\overline{\Omega})$. Now, in the domain $\Pi\defn\{(x,y)\,;\,y<y_0\}$, $\phi$ coincides 
with the harmonic extension of $\psi$, by uniqueness of the harmonic extension. 
Since $\Pi$ is invariant by translation in~$x$, we can compute the latter function by 
using the Fourier transform in~$x$. It results that, 
\be\label{n3000}
\forall x\in \xT^{d},\quad \forall y<y_0, \qquad 
\phi(x,y)=(e^{(y-y_0)\la D_x\ra}\psi)(x).
\ee
(Here, for $\tau<0$, $e^{\tau\la D_x\ra}$ denotes the Fourier 
multiplier with symbol $e^{\tau\la\xi\ra}$.) Indeed, 
the function $(e^{(y-y_0)\la D_x\ra}\psi)(x)$ is clearly harmonic and is equal to $\psi$ on $\{y=y_0\}$. 
Then, for $\la\alpha\ra+\beta>0$, it easily follows from~\e{n3000} 
and the Plancherel theorem that $\partial_x^\alpha\partial_y^\beta\phi$ belongs to $L^2(\Pi)$. 
On the other hand, on the strip $\{(x,y)\,;\, y_0<y<h(x)\}$, the function $\partial_x^\alpha\partial_y^\beta\phi$ 
is bounded and hence square integrable. 
By combining the two previous results, we obtain 
that 
$\partial_x^\alpha\partial_y^\beta\phi$ belongs to~$L^2(\Omega)$. To prove 
the second half of \e{decaytozero}, we use again the formula~\e{n3000} 
and the Plancherel theorem, to infer that $\partial_x^\alpha\partial_y^\beta\phi(\cdot,y)$ 
converges to $0$ in any Sobolev space $H^\mu(\xT^{d})$ ($\mu\ge 0$) when $y$ goes to $-\infty$. 
The desired decay result now follows from the Sobolev embedding $H^\mu(\xT^{d})\subset L^\infty(\xT^{d})$ for $\mu>d/2$.
\end{proof}

Let us fix some notations used in the rest of this section. 
Now, we consider a smooth function $h=h(x)$ in $C^\infty(\xT^d)$ and set
$$
\Omega=\{(x,y)\in\xT^{d}\times\xR\,:\,y<h(x)\}.
$$
We denote by $\varphi$ the harmonic extension of $h$ in $\overline{\Omega}$. This is the solution to 
\e{defi:varphi2-zero} in the special case where $\zeta=h$. Namely, $\varphi$ solves
\be\label{defi:varphi2}
\left\{
\begin{aligned}
&\Delta_{x,y}\varphi=0 \quad\text{in }\Omega,\\
&\varphi(x,h(x))=h(x) \text{ for all }x\in\xT^{d}.
\end{aligned}
\right.
\ee
Introduce $Q\colon \overline{\Omega}\to\xR$ defined by
$$
Q(x,y)=\varphi(x,y)-y.
$$
We call $Q$ the pressure. 
In this paragraph we gather some results for the pressure which are all 
consequences of the maximum principle. 
For further references, the main result states that $\partial_y Q<0$ everywhere in the fluid.

\begin{proposition}\label{Prop:p3.2}
\begin{enumerate}[i)]
\item\label{regP1} On the free surface $\Sigma=\{y=h(x)\}$, 
the function $Q$ satisfies the following properties:
\be\label{n8}
\partial_n Q=-\la \nabla_{x,y} Q\ra \quad\text{and}\quad n=-\frac{\nabla_{x,y}Q}{\la \nabla_{x,y} Q\ra},
\ee
where $n$ denotes the normal to $\Sigma$, given by
\be\label{n5}
n=\frac{1}{\sqrt{1+|\nabla h|^2}} \begin{pmatrix} -\nabla h \\ 1 \end{pmatrix}.
\ee
Moreover, the Taylor coefficient $a$ defined by
\be\label{defi:Taylor}
a(x)=-\partial_y Q(x,h(x)),
\ee
satisfies $a(x)>0$ for all $x\in \xT^{d}$. 
\item\label{regP2} For all $(x,y)$ in $\overline{\Omega}$, there holds
\be\label{n209}
\partial_y Q(x,y)<0.
\ee
Furthermore, 
\be\label{n210}
\inf_{\overline{\Omega}}(-\partial_y Q)\ge \min \Big\{ \inf_{x\in\xT^{d}}a(x),1\Big\}.
\ee

\item\label{regP3} The function $\la \nabla_{x,y} Q\ra$ belongs to $C^\infty(\overline{\Omega})$. 

\item \label{regP4} We have the following bound:
\be\label{esti:final8}
\sup_{(x,y)\in\overline{\Omega}}\la \nabla_{x,y}Q(x,y)\ra^2\le  \max_{\xT^d}\frac{(1-G(h)h)^2}{1+|\nabla_xh|^2}. 
\ee
\end{enumerate}
\end{proposition}
\begin{remark}\label{Rema:final1}
Consider the evolution problem for the Hele-Shaw equation $\partial_th+G(h)h=0$. Then 
in \cite{AMS} it is proved that 
$$
\inf_{x\in\xT^{d}}a(t,x)\ge \inf_{x\in\xT^{d}}a(0,x),\quad 
\sup_{x\in\xT^{d}}\la G(h)h(t,x)\ra\le \sup_{x\in\xT^{d}}\la G(h)h(0,x)\ra.
$$
Therefore, \e{n210} and \e{esti:final8} give two different control 
of the derivatives of the pressure, which are uniform in time.  
\end{remark}
\begin{proof}
In this proof, it is convenient to truncate the domain $\Omega$ to work with a compact domain. Consider 
$\beta>0$ such that the line $\xT^{d}\times \{-\beta\}$ is located underneath the free surface $\Sigma=\{y=h(x)\}$ and set
$$
\Omega_\beta=\{(x,y)\in\xT^{d}\times\xR\,;\,-\beta<y<h(x)\}.
$$
We will apply the maximum principle in $\Omega_\beta$ and then let $\beta$ goes to $+\infty$. 

$\ref{regP1})$ This point is well-known in certain communities, 
but we recall the proof for the reader's convenience. 
We begin by observing that, 
since $Q\ah=0$, on the free surface we have $\la \nabla_{x,y} Q\ra=\la \partial_n Q\ra$. 
So to prove that $\partial_n Q=-\la \nabla_{x,y} Q\ra$, 
it remains only to prove that $\partial_n Q\le 0$. 
To do so, we begin by noticing that $Q$ is solution 
to the following elliptic problem
$$
\Delta_{x,y}Q=0,\quad Q\ah=0. 
$$
We will apply the maximum principle in $\Omega_\beta$ with $\beta$ large enough. 
In view of \e{decaytozero}, there is $\beta>0$ such that 
$$
\forall y\le -\frac{\beta}{2},\qquad \lA \partial_y\varphi(\cdot,y)\rA_{L^\infty(\xT^{d})}\le \mez.
$$
In particular, on $\{y=-\beta\}$, there holds 
\be\label{Qincreases}
\forall y\le -\frac{\beta}{2},\quad\forall x\in\xT^{d},\qquad 
\partial_y Q(x,y)=\partial_y\varphi(x,y)-1\le -\mez.
\ee
On the other hand, by using the classical maximum principle for harmonic functions in $\Omega_\beta$, we see that 
$Q$ reaches its minimum on the boundary $\partial\Omega_\beta$. 
In light of \e{Qincreases}, the minimum is not attained on $\{y=-\beta\}$, 
so it is attained on~$\Sigma$. Since 
$Q$ vanishes there, this means that $Q\ge 0$ in $\Omega_\beta$. This immediately 
implies the wanted result $\partial_n Q\le 0$. 
In addition, since the boundary is smooth, 
we can apply 
the classical Hopf--Zaremba's principle to infer that $\partial_n Q<0$ on $\Sigma$. 

Let us now prove that $a>0$. Recall that, by notation, $\nabla$ denotes 
the gradient with respect to the 
horizontal variable only, $\nabla=(\partial_{x_1},\ldots,\partial_{x_d})^{t}$. 
Since $Q$ vanishes on $\Sigma$, 
we have
\be\label{n3001}
0=\nabla\big(Q\arrowvert_{y=h}\big)=(\nabla Q)\arrowvert_{y=h}+(\partial_yQ)\arrowvert_{y=h}\nabla h,
\ee
which implies that, on $y=h$ we have,
\begin{align}
a&=-(\partial_y Q)\arrowvert_{y=h}=-\frac{1}{1+|\nabla h|^2}
\Big(\partial_yQ\arrowvert_{y=h}-\nabla h\cdot (\nabla Q)\arrowvert_{y=h}\Big)\label{esti:final9}\\
&=-\frac{1}{\sqrt{1+|\nabla h|^2}}\notag
\partial_n Q.
\end{align}
Since $\partial_n Q<0$ on $\Sigma$, this implies that $a$ is a positive function. Eventually, 
remembering that 
$n=\frac{1}{\sqrt{1+|\nabla h|^2}} \left(\begin{smallmatrix} -\nabla h \\ 1 \end{smallmatrix}\right)$ 
and using~\e{n3001}, we verify that
$$
n=-\frac{\nabla_{x,y} Q}{\la \nabla_{x,y} Q\ra}\cdot
$$
This completes the proof of 
statement~~$\ref{regP1})$.

$\ref{regP2})$ Since the function $-\partial_y Q$ is 
harmonic in $\Omega$, 
the maximum principle applied in $\Omega_\beta$ implies that $-\partial_yQ$ reaches 
is minimum on the boundary $\partial\Omega_\beta$, so
$$
-\partial_y Q\ge \min\Big\{ \inf_{\Sigma} (-\partial_y Q),\inf_{\{y=-\beta\}}(-\partial_y Q)\Big\}.
$$
By letting $\beta$ goes to $+\infty$, we obtain~\e{n210} since $-\partial_y Q$ 
converges to $1$ (see~\e{decaytozero} applied with $\alpha=0$ and $\beta=1$). 
This in turn implies~\e{n209} 
in view of the fact that $a>0$, as proved in the previous point.

$\ref{regP3})$ Since we assume that $h$ is smooth, 
the function $Q$ belongs to $C^\infty(\overline{\Omega})$. 
As a consequence, to prove that $\la \nabla_{x,y} Q\ra$ is smooth, 
it is sufficient to prove that $\la\nabla_{x,y} Q\ra^2$ 
is bounded from below by a positive constant, 
which is an immediate consequence of~\e{n209}.

$\ref{regP4})$ Since $Q$ is an harmonic function, we have
$$
\Delta_{x,y}\la \nabla_{x,y}Q\ra^2=2\la\nabla_{x,y}^2Q\ra\ge 0. 
$$
Consequently, the maximum principle for sub-harmonic functions implies that
$$
\sup_{\overline{\Omega_\beta}}\la \nabla_{x,y}Q\ra^2= \sup_{\partial\Omega_\beta}\la \nabla_{x,y}Q\ra^2,
$$
where $\Omega_\beta$ is as above. By letting $\beta$ goes to $+\infty$, we obtain that
\be\label{esti:final8.1}
\sup_{\overline{\Omega}}\la \nabla_{x,y}Q\ra^2= \max\left\{\sup_{\Sigma}\la \nabla_{x,y}Q\ra^2,1\right\},
\ee
where we used as above that $\la\nabla_{x,y}Q\ra$ tends to $1$ when $y$ goes to $-\infty$. 
We are thus reduced to estimating $\la \nabla_{x,y}Q\ra^2$ on $\Sigma$. To do so, 
observe that the identity~\e{n3001} implies that, on $\Sigma$, we have
\be\label{esti:final8.2}
\la \nabla_{x,y}Q\ra^2=(1+|\nabla h|^2)(\partial_yQ)^2=(1+|\nabla h|^2)a^2. 
\ee
Using the computations already performed in~\e{esti:final9} and remembering that $Q=\varphi-y$, we obtain
$$
a=-\frac{1}{1+|\nabla h|^2}
\Big(-1+\partial_y\varphi\arrowvert_{y=h}-\nabla h\cdot (\nabla \varphi)\arrowvert_{y=h}\Big).
$$
On the other, since $\varphi$ is the harmonic extension of $h$, by 
definition of the Dirichlet-to-Neumann operator $G(h)$, one has
$$
G(h)h=\partial_y\varphi\arrowvert_{y=h}-\nabla h\cdot (\nabla \varphi)\arrowvert_{y=h}.
$$
We conclude that
$$
a=\frac{1-G(h)h}{1+|\nabla h|^2},
$$
which in turn implies that
$$
(1+|\nabla h|^2)a^2=\frac{(1-G(h)h)^2}{1+|\nabla_xh|^2}.
$$
By combining this with~\e{esti:final8.1} and \e{esti:final8.2}, 
we conclude the proof of statement~$\ref{regP4})$.
\end{proof}

\subsection{The key functional identity}\label{S:J(h)}

Let us recall some notations: 
we denote by $\kappa$ the mean curvature
\be\label{n6}
\kappa=-\cnx \left(\frac{\nabla h}{\sqrt{1+|\nabla h|^2}}\right).
\ee
Also, we denote by $\varphi=\varphi(x,y)$ the harmonic 
extension of $h$ in $\Omega$ given by~\e{defi:varphi2} and we use the notation
$$
Q(x,y)=\varphi(x,y)-y.
$$

\begin{proposition}\label{P:Positive2}
Let $d\ge 1$, assume that $h\colon \xT^{d}\to \xR$ is a smooth function and set
$$
J(h)\defn \int_{\xT^{d}} \kappa\,  G(h)h\dx.
$$
Then
\be\label{n11}
J(h)=\iint_{\Omega}\frac{\la \nabla_{x,y} Q\ra^2\la \nabla_{x,y}^2Q\ra^2
-\la \nabla_{x,y} Q\cdot \nabla_{x,y} \nabla_{x,y} Q\ra^2}{\la \nabla_{x,y} Q\ra^3}\dydx\ge 0.
\ee
\end{proposition}
\begin{remark}\label{rema:34}
\begin{enumerate}[i)]
\item Since $\la\nabla_{x,y}Q\ra\ge \la\partial_yQ\ra$, it follows 
from~\e{n209} and the positivity of the Taylor coefficient $a$ (see statement $\ref{regP1})$ 
in Proposition~\ref{Prop:p3.2}) ~that $\la \nabla_{x,y} Q\ra$ 
is bounded by a positive constant on $\overline{\Omega}$. 
On the other hand, directly from \e{decaytozero}, the function 
$\la \nabla_{x,y} Q\ra^2\la \nabla_{x,y}^2Q\ra^2-\la \nabla_{x,y} Q\cdot \nabla_{x,y} \nabla_{x,y} Q\ra^2$ 
belongs to $L^2(\Omega)$. 
It follows that the right-hand side of \e{n11} is a well defined integral.
\item To clarify notations, set $\partial_i=\partial_{x_i}$ for $1\le i\le n$ and $\partial_{n+1}=\partial_y$. 
Then 
\begin{equation*}
\left\{
\begin{aligned}
&\la \nabla_{x,y}^2 Q\ra^2=\sum_{1\le i,j\le n+1}(\partial_{i}\partial_{j}Q)^2,\\
&\la \nabla_{x,y} Q\cdot\nabla_{x,y}\nabla_{x,y} Q\ra^2
=\sum_{1\le i\le n+1}\biggl(\sum_{1\le j\le n+1}(\partial_{j}Q)\partial_{i}\partial_{j}Q\biggr)^2.
\end{aligned}
\right.
\end{equation*}
So, it follows from the Cauchy-Schwarz inequality that
\be\label{n2001}
\la \nabla_{x,y} Q\cdot\nabla_{x,y}\nabla_{x,y} Q\ra^2\le \la\nabla_{x,y} Q\ra^2\la \nabla_{x,y}^2 Q\ra^2.
\ee
This shows that $J(h)\ge 0$. 
\item If $d=1$, then one can simplify the previous expression. Remembering that 
$\Delta_{x,y}Q=0$, one can verify that
$$
J(h)=\mez \iint_\Omega\frac{\la \nabla_{x,y}^2Q\ra^2}{\la \nabla_{x,y} Q\ra}\dydx.
$$
Notice that, for the Hele-Shaw equation, one has a uniform in time estimate for 
$\la \nabla_{x,y} Q\ra$ as explained in Remark~\ref{Rema:final1}. Consequently, 
$J(h)$ controls the $L^2$-norm of the second-order derivative of $Q$. 
\end{enumerate}
\end{remark}
\begin{proof}
To prove Proposition~\ref{P:Positive2}, the main identity is given by the following result.
\begin{lemma}
There holds
\be\label{n369}
J(h)=\int_\Sigma \partial_n \la \nabla_{x,y} Q\ra\dHm,
\ee
where $\Sigma=\{y=h(x)\}$.
\end{lemma}
\begin{proof}
By definition of the Dirichlet-to-Neumann operator, one has
$$
G(h)h=\sqrt{1+|\nabla h|^2}\partial_n \varphi\ah,
$$
so
$$
\int_{\xT^{d}} \kappa\,  G(h)h\dx
=\int_{\xT^{d}} \kappa\,\partial_n \varphi\sqrt{1+|\nabla h|^2}\dx.
$$
Using the expression~\e{n5} for the normal $n$, we observe that 
$$
\partial_nQ=\partial_n \varphi-\frac{1}{\sqrt{1+|\nabla h|^2}}.
$$
Directly from the definition~\e{n6} of $\kappa$, we get that
$$
\int_{\xT^{d}}\kappa \dx=0.
$$
So by combining the previous identities, we deduce that
$$
J(h)
=\int_{\xT^{d}} \kappa\,(\partial_n Q)\ah \sqrt{1+|\nabla h|^2}\dx,
$$
which can be written under the form
\be\label{n358}
J(h)=\int_{\Sigma}\kappa \, \partial_n Q \dHm.
\ee

For this proof only, 
to simplify notations, we will write simply $\nabla$ and $\Delta$ 
instead of $\nabla_{x,y}$ and $\Delta_{x,y}$. 
Now, we recall from Proposition~\ref{Prop:p3.2} that, 
on the free surface $\Sigma$, we have
\be\label{n8bis}
\partial_n Q=-\la \nabla Q\ra \quad\text{and}\quad n=-\frac{\nabla  Q}{\la \nabla Q\ra}.
\ee
It follows that
$$
\kappa=-\cnxy \left(\frac{\nabla  Q}{\la \nabla Q\ra}\right),
$$
and
\be\label{n10}
\int_{\Sigma}\kappa \, \partial_n Q \dHm
=\int_{\Sigma}\cnx\left(\frac{\nabla Q}{\la \nabla Q\ra} \right)\la \nabla Q\ra \dHm.
\ee
Remembering that $\cnx \nabla Q=0$, one can further simplify:
\begin{align*}
\cnx \left(\frac{\nabla Q}{\la \nabla Q\ra} \right) \la \nabla Q\ra
&=\cnx\left(\frac{\nabla Q}{\la \nabla Q\ra} \la \nabla Q\ra \right)
-\frac{\nabla Q}{\la \nabla Q\ra}\cdot \nabla \la\nabla Q\ra\\
&=\cnx \nabla Q-\frac{\nabla Q}{\la \nabla Q\ra}\cdot \nabla \la\nabla Q\ra\\
&=-\frac{\nabla Q}{\la \nabla Q\ra}\cdot \nabla \la\nabla Q\ra\cdot
\end{align*}
Now, we use again the identity $n=-\frac{\nabla Q}{\la \nabla Q\ra}$ 
to infer that, on $\Sigma$, we have
$$
\cnx\left(\frac{\nabla Q}{\la \nabla Q\ra} \right) \la \nabla Q\ra
=n\cdot \nabla \la\nabla Q\ra=\partial_n \la \nabla Q\ra.
$$
Consequently, it follows from~~\e{n358} and \e{n10} that
$$
J(h)=\int_{\Sigma} \partial_n\la \nabla Q\ra\dHm.
$$
This completes the proof of the lemma.
\end{proof}
We have proved that $J(h)$ is equal to the integral over $\Sigma$ of 
$\partial_n\la \nabla Q\ra$. This suggests to apply the Stokes' theorem. 
To do so, as in the proof of Proposition~\ref{Prop:p3.2}, it is convenient 
to truncate the domain $\Omega$ to work with a compact domain. 
Again, we consider 
$\beta>0$ such that the hyperplane $\{y=-\beta\}$ is located 
underneath the free surface $\Sigma$ and set
$$
\Omega_\beta=\{(x,y)\in\xT^{d}\times\xT\,;\,-\beta<y<h(x)\}.
$$
Let us check that the contribution from the fictitious bottom disappears when $\beta$ goes to $+\infty$. 
\begin{lemma}
Denote by $\Gamma_\beta$ the bottom $\Gamma_\beta=\{(x,y)\in \xT^{d}\times\xR\,;\;y=-\beta\}$. Then
\be\label{n370}
\lim_{\beta\to+\infty}\int_{\Gamma_\beta} \partial_n\la \nabla Q\ra\dHm=0.
\ee
\end{lemma}
\begin{proof}
We have 
$$
\int_{\Gamma_\beta} \partial_n\la \nabla Q\ra\dHm
=-\int_{\xT^{d}}\partial_y \la \nabla Q\ra\dx
=-\int_{\xT^{d}}\frac{\nabla_x Q\cdot\nabla_x\partial_yQ+\partial_y Q\partial_y^2Q}{\la \nabla Q\ra}\dx.
$$
As we have seen in Remark~\ref{rema:34}, 
the function $\la \nabla Q\ra$ is bounded from below 
by a positive constant in $\Omega$. Consequently, it is bounded 
from below on $\Gamma_\beta$ uniformly with respect to $\beta$. 
On the other hand, it follows from \e{decaytozero} that
$$
\lim_{\beta\to+\infty}\lA (\nabla_x Q\cdot\nabla_x\partial_yQ
+\partial_y Q\partial_y^2Q)(\cdot,-\beta)\rA_{L^\infty(\xT^{d})}=0.
$$ 
This immediately gives the wanted result. 
\end{proof}

Now, we are in position to conclude the proof. It follows from \e{n369} that 
$$
J(h)=\int_{\partial\Omega_\beta} \partial_n\la \nabla Q\ra\dHm-\int_{\Gamma_\beta} \partial_n\la \nabla Q\ra\dHm.
$$
Now, remembering that $\la \nabla Q\ra$ belongs to $C^\infty(\overline{\Omega})$ 
(see statement~$\ref{regP3})$ in Proposition~\ref{Prop:p3.2}), one may apply 
the Stokes' theorem to infer that
$$
J(h)=\int_{\Omega_\beta} \Delta\la \nabla Q\ra\dydx-\int_{\Gamma_\beta} \partial_n\la \nabla Q\ra\dHm.
$$
Since $\la \nabla Q\ra>0$ belongs to $C^\infty(\overline{\Omega})$, one can compute $\Delta\la\nabla Q\ra$. 
To do so, we apply the general identity
$$
\Delta u^2=2u\Delta u+2\la \nabla u\ra^2,
$$
with $u=\la\nabla Q\ra$. 
This gives that
\begin{align*}
\Delta \la \nabla Q\ra&=\frac{1}{2\la \nabla Q\ra}\Big(\Delta \la \nabla Q\ra^2-2\la \nabla \la\nabla Q\ra\ra^2\Big)\\
&=\frac{1}{2\la \nabla Q\ra}\bigg(\Delta \la \nabla Q\ra^2-2\frac{\la \nabla Q\cdot\nabla\nabla Q\ra^2}{\la \nabla Q\ra^2}\bigg).
\end{align*}
On the other hand, since $\Delta Q=0$, one has
$$
\Delta \la \nabla Q\ra^2=\sum_{1\le j,k\le n+1}\partial_j^2(\partial_kQ)^2
=2\sum_{1\le j,k\le n+1}(\partial_j\partial_k Q)^2=2\la \nabla^2Q\ra^2.
$$
By combining the two previous identities, we conclude that
\begin{align*}
\Delta\la\nabla Q\ra&=\frac{1}{\la \nabla Q\ra^3}\Big(\la \nabla Q\ra^2\la \nabla^2Q\ra^2-\la \nabla Q\cdot \nabla \nabla Q\ra^2\Big).
\end{align*}

As we have seen in Remark~\ref{rema:34}, the previous term is 
integrable on $\Omega$. So, we can use the dominated convergence theorem and let $\beta$ goes to $+\infty$. 
Then~\e{n370} implies that the contribution from the bottom disappears 
from the limit and we obtain the wanted result~\e{n11}. 
This completes the proof.
\end{proof}

\subsection{Proof of Theorem~\ref{T1}}

We are now ready to prove Theorem~\ref{T1}. 
Let $(g,\mu)\in [0,+\infty)^2$ and assume that $h$ is a smooth solution to 
$$
\partial_{t}h+G(h)(gh+\mu \kappa)=0.
$$
We want to prove that
$$
\fract \int_{\xT^{d}} h^2\dx\le 0 \quad 
\text{and}\quad \fract \Hm(\Sigma)\le 0. 
$$

Multiplying the equation 
$\partial_{t}h+G(h)(gh+\mu\kappa)=0$ by $h$ and 
integrating over $\xT^{d}$, one obtains that
\be\label{n157}
\mez\fract \int_{\xT^{d}}h^2\dx
=-g\int_{\xT^{d}}hG(h)h\dx-\mu\int_{\xT^{d}}h G(h)\kappa\dx.
\ee
The first term in the right-hand side is non-positive 
since $G(h)$ is a non-negative operator. Indeed, as we recalled in the introduction, 
considering an arbitrary function $\psi$ and denoting by $\varphi$ its 
harmonic extension, 
it follows from Stokes' theorem that
\be\label{positivityDN}
\int_{\xT^{d}} \psi G(h)\psi\dx=\int_{\partial\Omega}\varphi \partial_n \varphi\diff\Hm=
\iint_{\Omega}\la\nabla_{x,y}\varphi\ra^2\dydx\ge 0.
\ee 
This proves that
$$
-g\int_{\xT^{d}}hG(h)h\dx\le 0.
$$
We now prove that the second term in the right-hand side of \e{n157} is also non-positive. To see this, 
we use~\e{n11} and the fact that~$G(h)$ is self-adjoint:
$$
\int_{\xT^{d}}h G(h)\kappa\dx=\int_{\xT^{d}}\kappa G(h)h\dx =J(h)\ge 0.
$$
This proves that
$$
\fract \int_{\xT^{d}} h^2\dx\le 0.
$$

It remains to prove that $\fract \Hm(\Sigma)\le 0$. Write
\begin{align*}
\fract \Hm(\Sigma)&=\fract \int_{\xT^{d}}\sqrt{1+|\nabla h|^2}\dx\\
&=\int_{\xT^{d}}\nabla_x (\partial_th) \cdot \frac{\nabla_x h}{\sqrt{1+|\nabla h|^2}}\dx\\
&=\int_{\xT^{d}} (\partial_th)\kappa \dx,
\end{align*}
to obtain
$$
\fract \Hm(\Sigma)=-\mu\int_{\xT^{d}}\kappa G(h)\kappa\dx-g J(h)\le 0,
$$
where we used again~\e{n11} and the property~\e{positivityDN} applied with $\psi=\kappa$. 

This completes the proof. 

\section{Strong decay for the Hele-Shaw equation}\label{S:J(h)decays}

In this section we prove Theorem~\ref{Theorem:J(h)decays} about 
the monotonicity of $J(h)$ for solutions of 
the Hele-Shaw equation. 
Recall that, by notation,
$$
J(h)=\int_{\xT^{d}} \kappa G(h)h\dx\quad\text{where}\quad
\kappa=-\cnx\left(\frac{\nabla h}{\sqrt{1+|\nabla h|^2}}\right).
$$
We want to prove that $J(h)$ is non-increasing under a mild-smallness assumption on $\nabla_{x,t}h$ at initial time. 

\begin{proposition}\label{LJ(h)I1}
Assume that 
$h$ is a smooth solution to the Hele-Shaw equation $\partial_t h+G(h)h=0$. Then
\be\label{n132}
\fract J(h)+\int_{\xT^d}\frac{\la\nabla\partial_t h\ra^2+\la\nabla^2 h\ra^2}{(1+|\nabla h|^2)^{3/2}}\dx
-\int_\xT\kappa \theta\dx\le 0,
\ee
where 
\be\label{defi:theta}
\theta=G(h)\left(\frac{\la\nabla_{t,x}h\ra^2}{1+|\nabla h|^2}\right)
-\cnx\left(\frac{\la\nabla_{t,x}h\ra^2}{1+|\nabla h|^2}\nabla h\right),
\ee
with $\la\nabla_{t,x}h\ra^2=(\partial_t h)^2+\la \nabla h\ra^2$. 
In addition, if $d=1$ then \e{n132} is in fact an equality.
\end{proposition}
\begin{proof}
If $h$ solves the Hele-Shaw equation $\partial_t h+G(h)h=0$, one can rewrite $J(h)$ under the form
$$
J(h)=-\int_{\xT^d} \kappa h_t \dx,
$$
where $h_t$ as a shorthand notation for $\partial_t h$.
Consequently, 
\be\label{esti:final10}
\fract J(h)+\int_{\xT^d}\kappa_t h_t\dx+\int_{\xT^d}\kappa h_{tt}\dx=0.
\ee
Let us compute the first integral. To do so, we use the Leibniz rule and then integrate
by parts, to obtain
\begin{align*}
\int_{\xT^d}\kappa_t h_t\dx
&=-\int_{\xT^d}\cnx\left(\frac{\nabla h_t}{\sqrt{1+|\nabla h|^2}}-\frac{\nabla h\cdot \nabla h_t}{(1+|\nabla h|^2)^{3/2}}\nabla h\right)h_t\dx\\
&=\int_{\xT^d}\frac{(1+|\nabla h|^2)\la \nabla h_t\ra^2-(\nabla h\cdot\nabla h_t)^2}{(1+|\nabla h|^2)^{3/2}}\dx.
\end{align*}
Now, the Cauchy-Schwarz inequality implies that
$$
(1+|\nabla h|^2)\la \nabla h_t\ra^2-(\nabla h\cdot\nabla h_t)^2\ge \la \nabla h_t\ra^2.
$$
(Notice that this is an equality in dimension $d=1$.) It follows from~\e{esti:final10} that
\be\label{n131}
\fract J(h)+\int_{\xT^d}\frac{\la \nabla h_t\ra^2}{(1+|\nabla h|^2)^{3/2}}\dx+\int_{\xT^d}\kappa h_{tt}\dx\le 0.
\ee

We now move to the most interesting part of the proof, which is the study the second term $\int \kappa h_{tt}$. 
The main idea is to use the fact that the Hele-Shaw equation can be written under the form 
of a modified Laplace equation. Let us pause to 
recall this argument introduced in~\cite{Aconvexity}. 
For the reader convenience, 
we begin by considering 
the linearized equation, which reads $\partial_t h+G(0)h=0$. Since the 
Dirichlet-to-Neumann operator $G(0)$ associated to a flat half-space 
is given by $G(0)=\lvert D\rvert$, that is the Fourier multiplier defined by 
$\lvert D\rvert e^{ix\cdot\xi}
=\lvert \xi\rvert e^{ix\cdot\xi}$, the linearized Hele-Shaw equation reads
$$
\partial_t h+\la D\ra h=0.
$$
Since $-\la D\ra^2=\Delta$, we find that
$$
\Delta_{t,x}h=\partial_t^2 h+\Delta h=0.
$$
The next result generalizes this observation to the Hele-Shaw equation. 

\begin{theorem}[from~\cite{Aconvexity}]\label{proposition:elliptic}
Consider a smooth solution $h$ to $\partial_t h+G(h)h=0$. Then 
\be\label{n131A}
\Delta_{t,x}h+B(h)^*\big( \la \nabla_{t,x}h\ra^2\big)=0,
\ee
where $B(h)^*$ is the adjoint (for the $L^2(\xT^d)$-scalar product) 
of the operator defined by
$$
B(h)\psi=\partial_y \mathcal{H}(\psi)\arrowvert_{y=h},
$$
where $\mathcal{H}(\psi)$ is the harmonic extension of $\psi$, solution to
$$
\Delta_{x,y}\mathcal{H}(\psi)=0\quad \text{in }\Omega,\qquad \mathcal{H}(\psi)\arrowvert_{y=h}=\psi.
$$
\end{theorem}
We next replace the operator $B(h)^*$ by an explicit expression which is easier to handle. 
Directly from the definition of $B(h)$ and the chain rule, one can check that (see for instance 
Proposition~$5.1$ in \cite{AMS}),
$$
B(h)\psi=\frac{G(h)\psi+\nabla h\cdot \nabla \psi}{1+|\nabla h|^2}\cdot
$$
Consequently,
$$
B(h)^*\psi=G(h)\left(\frac{\psi}{1+|\nabla h|^2}\right)
-\cnx\left(\frac{\psi}{1+|\nabla h|^2}\nabla h\right).
$$
It follows that 
\be\label{n131AA}
B(h)^*\big( \la \nabla_{t,x}h\ra^2\big)=\theta,
\ee
where $\theta$ is as defined in the statement of Proposition~\ref{LJ(h)I1}.

We now go back to the second term in the right-hand side of~\e{n131} and write that
$$
\int_{\xT^d}\kappa h_{tt}\dx=\int_{\xT^d}\kappa \Delta_{t,x}h\dx-\int_{\xT^d}\kappa \Delta h\dx.
$$
(To clarify notations, recall that $\Delta$ denotes the Laplacian with respect to the variable $x$ only.) 
By plugging this in~\e{n131} and using \e{n131A}--\e{n131AA},
we get
$$
\fract J(h)+\int_{\xT^d}\frac{\la \nabla h_t\ra^2}{(1+|\nabla h|^2)^{3/2}}\dx
-\int_\xT\kappa \theta\dx
-\int_{\xT^d}\kappa \Delta h\dx\le 0.
$$
As a result, to complete the proof of Proposition~\ref{LJ(h)I1}, it remains only to show that
\be\label{claim:kappahxx}
-\int_{\xT^d}\kappa \Delta h\dx\ge\int_{\xT^d}\frac{|\nabla^2 h|^2}{(1+|\nabla h|^2)^{3/2}}\dx.
\ee
Notice that in dimension $d=1$, we have
$$
\kappa=-\frac{\partial_x^2 h}{(1+(\partial_xh)^2)^{3/2}},
$$
so~\e{claim:kappahxx} is in fact 
an equality. To prove~\e{claim:kappahxx} in arbitrary dimension, 
we begin by applying the Leibniz rule to write
\be\label{n4001}
-\kappa=\frac{\Delta h}{\sqrt{1+|\nabla h|^2}}-\frac{\nabla h\otimes \nabla h:\nabla^2 h}{(1+|\nabla h|^2)^{3/2}},
\ee
where we use the standard notations $\nabla h\otimes \nabla h=((\partial_ih)(\partial_j h))_{1\le i,j\le d}$, 
$\nabla ^2h=(\partial_i\partial_j h)_{1\le i,j\le d}$ together with $A:B=\sum_{i,j}a_{ij}b_{ij}$. So,
\be\label{n147}
-\int_{\xT^d}\kappa \Delta h\dx=
\int_{\xT^d}\frac{(\Delta h)^2}{\sqrt{1+|\nabla h|^2}}\dx
-\int_{\xT^d}\frac{(\Delta h)\nabla h\otimes \nabla h:\nabla^2 h}{(1+|\nabla h|^2)^{3/2}}\dx.
\ee
On the other hand, 
by integrating by parts twice, we get
\begin{align*}
&\int_{\xT^d}\frac{(\Delta h)^2}{\sqrt{1+|\nabla h|^2}}\dx=\\
&\qquad=\sum_{i,j}\int_{\xT^d}\frac{(\partial_i^2 h)(\partial_j^2h)}{\sqrt{1+|\nabla h|^2}}\dx\\
&\qquad=\sum_{i,j}\int_{\xT^d}\frac{(\partial_i\partial_j h)^2}{\sqrt{1+|\nabla h|^2}}\dx\\
&\qquad\quad
+\sum_{i,j,k}\frac{(\partial_ih)(\partial_k h)(\partial_j^2h)(\partial_i\partial_kh)-(\partial_ih)(\partial_kh)(\partial_{i}\partial_{j}h)(\partial_j\partial_kh)}{(1+|\nabla h|^2)^{3/2}}\dx\\
&\qquad=\int_{\xT^d}\frac{(1+|\nabla h|^2)\la\nabla ^2h\ra^2+(\Delta h)\nabla h\otimes \nabla h:\nabla^2 h-(\nabla h\cdot \nabla^2h)^2
}{(1+|\nabla h|^2)^{3/2}}\dx.\\
\end{align*}
By combining this with~\e{n147} and simplifying, we obtain
$$
-\int_{\xT^d}\kappa \Delta h\dx=\int_{\xT^d}\frac{(1+|\nabla h|^2)\la\nabla ^2h\ra^2-(\nabla h\cdot\nabla\nabla h)^2}{(1+|\nabla h|^2)^{3/2}}\dx.
$$
Now, by using the Cauchy-Schwarz inequality in $\xR^d$, we obtain the wanted inequality~\e{claim:kappahxx}, and 
the proposition follows.
\end{proof}

In view of the previous proposition, to prove that $J(h)$ is non-increasing, 
it remains to show that the last term in the 
left-hand side of~\e{n132} can be absorbed by the second one. 
It does not seem feasible to get such a result by exploiting some special identity for the solutions, 
but, as we will see, we do have an inequality 
which holds under a very mild smallness assumption. 
We begin by making a  
smallness assumption on the space and time 
derivatives of the unknown~$h$. We 
will next apply a maximum principle to bound these derivatives 
in terms of the initial data only.

\begin{lemma}\label{Lemma:L938}
Let $c_d<1/2$ and assume that 
\begin{equation}\label{assu:L938}
\sup_{t,x}\la \nabla h(t,x)\ra^2 \le c_d,\quad 
\sup_{t,x} ( h_t(t,x))^2\le c_d.
\end{equation}
Then
\be\label{n141}
\int_{\xT^d}\kappa \theta\dx
\le \gamma_d \int_{\xT^d}\frac{\la\nabla h_t\ra^2
+\la \nabla^2h\ra^2}{(1+|\nabla h|^2)^{3/2}}\dx
\ee
with
$$
\gamma_d = 2 c_d \left(d+\left(d+\sqrt{d}\right) c_d\right) + 4 \left(c_d\left(d+ (d+1) c_d\right)\left(\frac{12}{1-2c_d}+1\right)\right)^{\mez}.
$$
\end{lemma}
\begin{proof}
To shorten notations, let us set
$$
H\defn \int_{\xT^d}\frac{\la\nabla h_t\ra^2
+\la \nabla^2h\ra^2}{(1+|\nabla h|^2)^{3/2}}\dx,
$$
and
$$
\zeta\defn\frac{\la\nabla_{t,x}h\ra^2}{1+|\nabla h|^2}
     =\frac{(\partial_th)^2+ \la\nabla h\ra^2}{1+|\nabla h|^2}.
$$
Then, by definition of $\theta$ (see~\e{defi:theta}), we have
$$
\theta=G(h)\zeta-\cnx (\zeta\nabla h)
= I_1 + I_2,
$$
with 
$$
I_1= - \zeta \Delta h
$$
and
$$
I_2 = G(h)\zeta - \nabla \zeta\cdot \nabla h.
$$
We will study the contributions of $I_1$ and $I_2$ to
$\int \kappa \theta\dx$ separately. 

{1) \em Contribution of $I_1$.} 
We claim that
\be\label{n4002}
-\int_{\xT^d}\kappa \zeta \Delta h\dx\le  
 \int_{\xT^d}\zeta \left(d+ (d+\sqrt d)|\nabla h|^2\right) \frac{|\nabla\nabla h|^2}{(1+|\nabla h|^2)^{3/2}}\dx.
\ee
To see this, we use again \e{n4001}, to write
$$
-\kappa\zeta \Delta h=
\zeta\frac{(\Delta h)^2}{\sqrt{1+|\nabla h|^2}}-\zeta
\frac{(\Delta h)\nabla h\otimes \nabla h:\nabla^2 h}{(1+|\nabla h|^2)^{3/2}}.
$$
Then we recall that for all $v\colon\xR^d \mapsto \xR^d$
$$
(\cnx v)^2 
=\sum_i\sum_j \partial_i v_i \partial_j v_j
\le \sum_i\sum_j \frac{1}{2} 
\bigl((\partial_i v_i)^2 + (\partial_j v_j)^2\bigr)
\le d |\nabla u|^2,
$$
and therefore
\begin{equation}\label{div}
(\Delta h)^2 \le d \,|\nabla\nabla h|^2.
\end{equation}
Then, by using the Cauchy-Schwarz inequality, we prove
the claim~\e{n4002}. 

Now, observe that, by definition of $\zeta$ we have $\zeta\le \la\nabla_{t,x}h\ra^2$. So, 
by assumption \e{assu:L938}, we deduce that
\begin{align*}
\zeta\left(d+ (d+\sqrt d)|\nabla h|^2\right)&\le \la \nabla_{t,x} h\ra^2\left(d+ (d+\sqrt d)|\nabla h|^2\right)\\
&\le 2c_d\left(d+ (d+\sqrt d)c_d\right).
\end{align*}
Therefore, it follows from \e{n4002} that 
\be\label{n4002ter}
-\int_{\xT^d}\kappa \zeta \Delta h\dx\le  2c_d\left(d+ (d+\sqrt d)c_d\right) H.
\ee

{2) \em Contribution of $I_2$.} 
We now estimate the quantity
\be\label{n143}
\int_{\xT^d} \kappa\, \big( G(h)\zeta-\kappa \nabla \zeta \cdot \nabla h\big)\dx.
\ee
We will prove that the absolute value of this term is bounded by
\be\label{esti:final1}
4 \left(c_d\left(d+ (d+1) c_d\right)\left(\frac{12}{1-2c_d}+1\right)\right)^{\mez}H.
\ee
By combining this estimate with \e{n4002ter}, this will imply the wanted inequality~\e{n141}. 

To begin, we apply the Cauchy-Schwarz inequality to bound the absolute value of~\e{n143} by
$$
\left(\int_{\xT^d}  (1+ |\nabla h|^2)^{3/2}\kappa^2\dx\right)^\mez 
\left(\int_{\xT^d} \frac{(G(h)\zeta
- \nabla \zeta\cdot \nabla h)^2}{(1+|\nabla h|^2)^{3/2}}\dx\right)^\mez. 
$$
We claim that 
\be\label{esti:final2}
\int_{\xT^d}  (1+ |\nabla h|^2)^{3/2}\kappa^2\dx\le 2\big(d+(d+1)c_d\big) H,
\ee
and
\be\label{esti:final3}
\int_{\xT^d} \frac{(G(h)\zeta
- \nabla \zeta\cdot \nabla h)^2}{(1+|\nabla h|^2)^{3/2}}\dx\le 
8c_d\left(\frac{12}{1-2c_d}+1\right)H.
\ee
It will follow from these claims that the absolute value of~\e{n143} is bounded by~\e{esti:final1}, which in turn will 
complete the proof of the lemma.

We begin by proving~\e{esti:final2}. 
Recall from~\e{n4001} that
\be\label{n4001bis}
-\kappa=\frac{\Delta h}{\sqrt{1+|\nabla h|^2}}-\frac{\nabla h\otimes \nabla h:\nabla^2 h}{(1+|\nabla h|^2)^{3/2}},
\ee
and therefore, using the inequality $(\Delta h)^2 \le d \,|\nabla\nabla h|^2$ (see~\eqref{div}),
$$
\kappa^2\le 
 2\big(d+(d+1)|\nabla h|^2\big)
\frac{|\nabla\nabla h|^2}{(1+|\nabla h|^2)^3},
$$
which implies~\e{esti:final2}, remembering that $\la \nabla h\ra^2\le c_d$, 
by assumption~\e{assu:L938}. 

We now move to the proof of~\e{esti:final3}. 
Since
$$
\frac{(G(h)\zeta
- \nabla \zeta\cdot \nabla h)^2}{(1+|\nabla h|^2)^{3/2}}\le 
2(G(h)\zeta)^2+2 \la \nabla \zeta\ra^2,
$$
it is sufficient to prove that
\be\label{n145}
\int_{\xT^d} (G(h)\zeta)^2\dx
+\int_{\xT^d} |\nabla \zeta|^2\dx
\le 4c_d\left(\frac{12}{1-2c_d}+1\right)H.
\ee
To establish~\e{n145}, the crucial point will be to bound 
the $L^2$-norm of $G(h)\zeta$ in terms of the $L^2$-norm of $\nabla \zeta$. 
Namely, we now want to prove the following estimate: 
if $|\nabla h|^2\le c_d$ with $c_d<1/2$, then
\be\label{d12}
\int_{\xT^d} (G(h)\zeta)^2\dx\le \frac{12}{1-2 c_d} 
\int_{\xT^d} |\nabla \zeta|^2\dx.
\ee
To do so, we use 
the following inequality\footnote{This inequality belongs to the family of Rellich type inequalities, 
which give a control on the boundary of the normal 
derivative in terms of the tangential one.}  (proved in Appendix~\ref{A:Rellich}):
\be\label{d10}
\int_{\xT^d} (G(h)\zeta)^2\dx \le   
\int_{\xT^d} (1+|\nabla h|^2)|\nabla \zeta-\mathcal{B} \nabla h|^2 \dx,
\ee
where
\be\label{d11}
\mathcal{B}=\frac{G(h)\zeta+\nabla \zeta \cdot \nabla h}{1+|\nabla h|^2}.
\ee
Then, by replacing $\mathcal{B}$ in \e{d10} by its expression \e{d11}, we obtain  
\begin{multline*}
\int_{\xT^d} (G(h)\zeta)^2\dx
\\ 
\le\int_{\xT^d} (1+|\nabla h|^2)
 \left|\frac{(1+|\nabla h|^2){\rm Id}-\nabla h\otimes \nabla h}
   {1+|\nabla h|^2}\nabla \zeta
 -\frac{G(h)\zeta}{1+|\nabla h|^2}\nabla h\right|^2 \dx.
\end{multline*}
So, expanding the right-hand side and simplyfying, we get
\begin{align*}
&\int_{\xT^d} \frac{1-|\nabla h|^2}{1+|\nabla h|^2} (G(h)\zeta)^2\dx\\
&\qquad=  \int_{\xT^d}
 \frac{|((1+|\nabla h|^2){\rm Id} -\nabla h \otimes \nabla h) \nabla \zeta|^2}{1+|\nabla h|^2} \dx\\
&\qquad\quad 
-2\int_{\xT^d}\nabla h \cdot \frac{\bigl(((1+|\nabla h|^2){\rm Id}-\nabla h \otimes \nabla h) \nabla \zeta\bigr)}{1+|\nabla h|^2} 
G(h)\zeta \dx.
\end{align*}
Hence, by using the Young inequality,
\begin{align*}
\int_{\xT^d} \frac{1-|\nabla h|^2}{1+|\nabla h|^2} (G(h)\zeta)^2\dx
&\le \int_{\xT^d} 
  \frac{\bigl|(1+|\nabla h|^2){\rm Id} -\nabla h\otimes \nabla h\bigr|^2}{1+\nabla h|^2}
  |\nabla \zeta|^2\dx\\
&\quad+\int_{\xT^d} \frac{|\nabla h|^2}{1+|\nabla h|^2} (G(h)\zeta)^2\dx\\
&\quad+\int_{\xT^d} \frac{|(1+|\nabla h|^2){\rm Id} -\nabla h\otimes\nabla h|^2}
 {1+|\nabla h|^2} |\nabla \zeta|^2\dx.
\end{align*}
Now we write 
$$
\frac{|((1+|\nabla h|^2){\rm Id} -\nabla h \otimes \nabla h) |^2}{1+|\nabla h|^2}\le
\frac{(1+2|\nabla h|^2)^2}{1+|\nabla h|^2},
$$
to obtain
$$
\int_\xT \frac{1-2|\nabla h|^2}{1+|\nabla h|^2} 
(G(h)\zeta)^2\dx\le 2\int_\xT 
   \frac{(1+2|\nabla h|^2)^2}{1+|\nabla h|^2}|\nabla \zeta|^2\dx.
$$
Now, recalling that $|\nabla h|^2 \le c_d < 1/2$, we get
$$
\int_\xT 
(G(h)\zeta)^2\dx\le 2 \frac{(1+c_d) (1+2c_d)^2 }{1-2 c_d}\int_\xT 
|\nabla \zeta|^2\dx\le \frac{12}{1-2 c_d}\int_\xT 
|\nabla \zeta|^2\dx.
$$

In view of~\e{d12}, to prove the wanted 
inequality~\e{n145}, we are reduced to establishing
$$
\int_\xT |\nabla \zeta|^2\dx\le 4 c_d 
\int_\xT\frac{|\nabla h_t|^2
         + |\nabla\nabla h|^2}{(1+|\nabla h|^2)^{3/2}}\dx.
$$
Since
$$
\nabla \zeta = 2\frac{h_t}{(1+|\nabla h|^2)^{1/4} } 
    \frac{\nabla h_t}{(1+|\nabla h|^2)^{3/4}}
    + 2\frac{(1-(h_t)^2)\nabla h}{(1+|\nabla h|^2)^{5/4}}
    \cdot \frac{\nabla\nabla h}{(1+|\nabla h|^2)^{3/4}},
$$
the latter inequality will be satisfied provided that
$$
\frac{\left((1-(h_t)^2)|\nabla h|\right)^2}{(1+|\nabla h|^2)^{5/2}}\le  c_d,
\quad 
\frac{\left( h_t\right)^2}{(1+|\nabla h|^2)^{1/2}}\le c_d
$$
The latter couple of conditions are obviously satisfied when
\be\label{n150}
|\nabla h|^2\le c_d,\quad |h_t|^2\le c_d\quad\text{with}\quad c_d< \mez.
\ee
This completes the proof of Lemma~\ref{Lemma:L938}.
\end{proof}

We are now in position to complete the proof. Recall that Proposition~\ref{LJ(h)I1} implies that
$$
\fract J(h)+\int_{\xT^d}\frac{\la\nabla\partial_t h\ra^2+\la\nabla^2 h\ra^2}{(1+|\nabla h|^2)^{3/2}}\dx
\le \int_\xT\kappa \theta\dx.
$$
On the other hand, Lemma~\ref{Lemma:L938} implies that
$$
\int_\xT\kappa \theta\dx\le \gamma_d 
\int_{\xT^d}\frac{\la\nabla\partial_t h\ra^2+\la\nabla^2 h\ra^2}{(1+|\nabla h|^2)^{3/2}}\dx,
$$
with
$$
\gamma_d = 2 c_d \left(d+\left(d+\sqrt{d}\right) c_d\right) + 4 \left(c_d\left(d+ (d+1) c_d\right)\left(\frac{12}{1-2c_d}+1\right)\right)^{\mez},
$$
provided 
\be\label{esti:final4}
\sup_{t,x}\la \nabla h(t,x)\ra^2 \le c_d,\quad 
\sup_{t,x} ( h_t(t,x))^2\le c_d.
\ee 
We now fix $c_d\in [0,1/4]$ by solving the equation $\gamma_d=1/2$ 
(the latter equation has a unique solution since $c_d\mapsto \gamma_d$ is strictly increasing). 
It follows that, 
\be\label{n152}
\fract J(h)+\mez \int_{\xT^d}\frac{\la\nabla\partial_t h\ra^2+\la\nabla^2 h\ra^2}{(1+|\nabla h|^2)^{3/2}}\dx\le 0.
\ee
The expected decay of $J(h)$ is thus seen to hold 
as long as the solution $h=h(t,x)$ satisfies the  assumption~\e{esti:final4}. Consequently, 
to conclude the proof of Theorem~\ref{Theorem:J(h)decays}, 
it remains only to show that the assumption~\e{esti:final4} on the solution will hold provided 
that it holds initially. To see this, we use the fact that 
there is a maximum 
principle for the Hele-Shaw equation, for 
space {\em and} time derivatives (the maximum principle for spatial 
derivatives is well-known (see~\cite{Kim-ARMA2003,ChangLaraGuillenSchwab,AMS}), 
the one for time derivative is given by Theorem~$2.11$ in~\cite{AMS}). 
This means that the assumption \e{n150} holds 
for all time $t\ge 0$ provided that it holds at time $0$. 
This concludes the proof of Theorem~\ref{Theorem:J(h)decays}.

\section{On the thin-film equation}

The aim of this section is two-fold. Firstly, for the reader's convenience, 
we collect various known results for the equation
$$
\partial_th-\cnx\big(h\nabla(gh-\mu\Delta h)\big)=0.
$$
Secondly, we study the decay of certain Lebesgue norms for the thin-film equation
\be\label{thinfilm-d}
\partial_t h +\cnx (h\nabla \Delta h) = 0.
\ee
Recall that we consider only smooth positive solutions. Then, 
since
$$
\partial_t \int_{\xT^{d}} h\dx=0,
$$
and since $h=|h|$, the $L^1$-norm is preserved and, obviously, 
it is a Lyapunov functional. 
We study more generally the decay of Lebesgue norms $\int h^p\dx$ with $p>0$. 
The study of this question 
goes back to the work of Beretta, Bertsch and Dal Passo~\cite{BerettaBDP-ARMA-1995} 
and was continued by Dal Passo, Garcke and Gr\"{u}n~\cite{DPGG-Siam-1998}, 
Bertsch, Dal Passo, Garcke and Gr\"{u}n~\cite{BDPGG-ADE-1998}. In these papers, 
it is proved that 
$$
\int_{\xT^d} h^{p}\dx,
$$ is a Lyapunov functional for $1/2< p< 2$ and $d=1$. More recently, 
J\"{u}ngel and Matthes performed in~\cite{JungelMatthes-Nonlinearity-2006} 
a systematic study of entropies 
for the thin-film equation, based on a computer assisted proof. They obtain the same result, allowing for the endpoints, 
that is for 
$1/2\le p\le 2$; they give a complete proof in space dimension $d=1$ and sketch the proof of the general case in Section $5$ of their paper. 
Here, we will not prove any new result, but we propose a new proof of the 
fact that $\int_{\xT^d} h^p\dx$ is a Lyapunov functional in any dimension $d\ge 1$ and for any $p$ in the closed interval $[1/2,2]$. 
Our proof is self-contained and elementary. This will allow us to introduce a functional 
inequality as well as some integration by parts arguments used lated to study the Boussinesq equation.

\subsection{Classical Lyapunov functionals}

\begin{proposition}\label{prop:lubrik1}
Let $(g,\mu)\in [0,+\infty)^2$ and assume that $h$ is a smooth positive solution to 
the thin-film equation
$$
\partial_th-\partial_x\big(h\partial_x(gh-\mu\partial_x^2 h)\big)=0.
$$
Then
$$
\fract \int_{\xT} h^2\dx \le 0\quad\text{and}\quad \fract\int_{\xT} (\partial_xh)^2\dx\le 0.
$$
\end{proposition}
\begin{proof}
Multiply the equation by $h$ and integrate by parts. Then
$$
\mez\fract\int_\xT h^2\dx+g\int_\xT hh_x^2\dx+\mu\int_\xT hh_xh_{xxx}\dx=0.
$$
Now notice that
$$
\int_\xT hh_xh_{xxx}\dx=-\int_\xT h_x^2h_{xx}\dx-\int_\xT hh_{xx}^2\dx=-\int_\xT hh_{xx}^2\dx.
$$
Consequently, 
$$
\mez\fract\int_\xT h^2\dx+\int_\xT h(gh_x^2+\mu h_{xx}^2)\dx= 0.
$$
Similarly, by multiplying the equation by $h_{xx}$ 
and integrating by parts, one obtains that
$$
\mez \fract \int_\xT h_x^2\dx+\int_\xT h\big( gh_{xx}^2+\mu h_{xxx}^2\big)\dx=0.
$$
Now, it follows directly from the assumption $h\ge 0$ that
$$
\fract \int_\xT h^2\dx\le 0 \quad\text{and}\quad \fract \int_\xT (\partial_x h)^2\dx\le 0.
$$
This completes the proof.
\end{proof}

Half of the previous results can be generalized to the $d$-dimensional case  
in a straightforward way.
\begin{proposition}\label{prop:lubrik1n}
Let $d\ge 1$. If $h$ is a smooth positive solution to 
the thin-film equation
\be\label{TFwith}
\partial_th+\cnx\big(h\nabla \Delta h)=0,
\ee
then
$$
\fract\int_{\xT^{d}} \la \nabla h\ra^2\dx\le 0.
$$
If $h$ is a smooth positive solution to 
\be\label{TFwithout}
\partial_th-\cnx\big(h\nabla h)=0,
\ee
then
$$
\fract \int_{\xT^{d}} h^2\dx \le 0.
$$
\end{proposition}
\begin{proof}
For the first point, we multiply the equation by $-\Delta h$ and integrate by parts. 
For the second point, we multiply the equation by $h$ and integrate by parts.
\end{proof}

This raises the question of proving the decay of the $L^2$-norm for \e{TFwith} 
(resp.\ the decay of the $L^2$-norm of $\nabla h$ for \e{TFwithout}) in arbitrary dimension. 
We study these questions in the rest of this section (resp.\ in Section~\ref{S:Boussinesq}).

\subsection{Decay of certain Lebesgue norms}

We begin by considering the special case of the $L^2$-norm. 
The interesting point is that, in this case, we are able to prove that it is a Lyapunov functional by 
means of a very simple argument. 
\begin{proposition}\label{prop:L2decaysagain}
Let $d\ge 1$ and consider a smooth positive solution $h$ to~$\partial_t h +\cnx (h\nabla \Delta h) = 0$. 
Then
\be\label{decayL2TF}
\fract \int_{\xT^{d}} h(t,x)^2\dx+\frac{2}{3}\int_{\xT^{d}} h |\nabla\nabla h|^2 \dx
+ \frac{1}{3} \int_{\xT^{d}} h |\Delta h|^2\dx = 0.
\ee
\end{proposition}
\begin{proof}
The energy identity reads
$$
\mez \fract \int_{\xT^{d}} h(t,x)^2\dx+I=0
$$
where
$$
I= - \int_{\xT^{d}} h \nabla h \cdot \nabla \Delta h\dx.
$$
Integrating by parts we get
$$
I= \int_{\xT^{d}} h|\Delta h|^2 \dx
      + \int_{\xT^{d}} |\nabla h|^2 \Delta h \dx= I_1+I_2.
$$
We integrate by parts again to rewrite $I_2$ under the form
$$
I_2 = - 2\int_{\xT^{d}} ((\nabla h \cdot \nabla) \nabla h) \cdot \nabla h\dx 
=  2 \int_{\xT^{d}} h |\nabla \nabla h|^2\dx
     - 2 I.$$
It follows that
$$
I = \frac{2}{3}\int_{\xT^{d}} h |\nabla\nabla h|^2 \dx+ \frac{1}{3} \int_{\xT^{d}} h |\Delta h|^2\dx,
$$
which is the wanted result.
\end{proof}

As explained in the paragraph at the beginning of this section, our main goal is to 
give a simple and self-contained proof of the fact that the quantities $\int_{\xT^d} h^p\dx$ are Lyapunov functionals 
for any $d\ge 1$ and any real number $p$ in $[1/2,2]$. To do so, 
the key ingredient will be 
a new functional inequality of independent interest which is given by the following
\begin{proposition}\label{P:refD.1v2}
For any $d\ge 1$, any real number $\mu$ and any bounded positive function $\theta$ in $H^2(\xT^{d})$,
\be\label{youpi2}
\frac{\mu^2}{3}\int_{\xT^{d}} \theta^{\mu-1}\big|\nabla \theta\big|^4\dx
\le 
  \int_{\xT^{d}} \theta^{\mu+1} (\Delta \theta)^2 \dx+
 2\int_{\xT^{d}} \theta^{\mu+1} (\nabla\nabla \theta)^2 \dx.
\ee
\end{proposition}
\begin{remark}\label{Rema:endpoint}
Dal Passo, Garcke and Gr\"un proved in \cite[Lemma~$1.3$]{DPGG-Siam-1998} the following identity:
\begin{align*}
&\int_{\xT^{d}}f'(\theta)\la\nabla \theta\ra^2\Delta \theta\dx\\
&\qquad=-\frac{1}{3}\int_{\xT^{d}}f''(\theta)\la \nabla \theta\ra^4\dx\\
&\qquad\quad+\frac{2}{3}\left(\int_{\xT^{d}}f(\theta)\la \nabla^2\theta\ra^2\dx-\int_{\xT^{d}}f(\theta)(\Delta \theta)^2\dx\right).
\end{align*}
Assuming that $\mu\neq -1$, by using this equality with $f(\theta)=\theta^\mu$, we obtain that
\begin{multline}\label{youpi2-vDPGG}
\int_{\xT^{d}} \theta^{\mu-1}\big|\nabla \theta\big|^4\dx\\
\qquad\qquad\le C(\mu)\left(
  \int_{\xT^{d}} \theta^{\mu+1} (\Delta \theta)^2 \dx+
 \int_{\xT^{d}} \theta^{\mu+1} (\nabla\nabla \theta)^2 \dx\right).
\end{multline}
So \e{youpi2} is a variant of the previous inequality which holds uniformly in $\mu$ 
(this means that we can consider the case $\mu=-1$ already encountered in Proposition~\ref{theo:logSob}, the latter case being 
important for the application since it is needed to control the $L^2$-norms).
\end{remark}
\begin{proof}
The result is obvious when $\mu=0$ so we assume $\mu\neq 0$ in the sequel. 
We then proceed as in the proof of Proposition~\ref{theo:logSob}. We begin by writing that
\begin{align*}
\int_{\xT^{d}}\theta^{\mu-1}\big|\nabla \theta\big|^4\dx
&=\int_{\xT^{d}}\frac{1}{\mu}\nabla \theta^\mu\cdot \nabla \theta \la \nabla\theta\ra^2\dx\\
&=-\frac{1}{\mu}\int_{\xT^{d}}\theta^\mu (\Delta \theta)\la \nabla \theta\ra^2\dx\\
&\quad-\frac{2}{\mu}\int_{\xT^{d}}\theta^\mu 
\nabla \theta\cdot\big[ (\nabla \theta\cdot\nabla )\nabla \theta\big]\dx.
\end{align*}
Then, by using Cauchy-Schwarz arguments similar to the ones used in the proof of Proposition~\ref{theo:logSob}, we infer that 
\begin{multline*}
\int_{\xT^{d}} \theta^{\mu-1}\big|\nabla \theta\big|^4\dx\\
\qquad\qquad\le \frac{1}{\mu^2}\left(
   \left(\int_{\xT^{d}} \theta^{\mu+1} (\Delta \theta)^2 \dx\right)^{\mez}+
 2\left(\int_{\xT^{d}} \theta^{\mu+1} (\nabla\nabla \theta)^2 \dx\right)^{\mez}\right)^2.
\end{multline*}
To conclude the proof, it remains only to use 
the elementary inequality $(x+2y)^2\le 3(x^2+2y^2)$. 
\end{proof}

We are now ready to give an elementary proof of the following result.

\begin{proposition}\label{positivity}
Consider a real number $m$ in $[-1/2,0)\cup(0,1]$. 
Then, for all smooth solution to $\partial_t h +\cnx (h\nabla \Delta h) = 0$, 
\be\label{wantedmn1}
\frac{1}{m(m+1)}\fract \int_{\xT^{d}} h^{m+1}\dx +C_m\int_{\xT^{d}}h^{m-2}\la\nabla h\ra^4\dx\le 0,
\ee
where
$$
C_m=\frac{1}{9}(-2m^2+m+1)\ge 0.
$$
\end{proposition}
\begin{proof}
We begin by multiplying the equation by $\frac{1}{m}h^m$ and 
integrating by parts, to get
\begin{align*}
&\frac{1}{m(m+1)}\fract \int_{\xT^{d}} h^{m+1}\dx+ P_m = 0\quad\text{where}\\
&P_m=-\frac{1}{m}\int_{\xT^{d}}\nabla h^m \cdot \big( h\nabla \Delta h\big)\dx
=-\int_{\xT^{d}}h^m\nabla h\cdot \nabla \Delta h\dx.
\end{align*}
Now, we use the following trick: there are two possible integrations by parts to 
compute an integral of the form
$\int f\nabla g\cdot \nabla \Delta h\dx$.
Indeed,
\begin{align*}
&\int_{\xT^{d}} f(\partial_i g)(\partial_i \partial_j^2) h\dx=-\int f (\partial_i^2 g)(\partial_j ^2h)\dx
-\int (\partial_if)(\partial_ig)\partial_j^2h\dx\\
&\int_{\xT^{d}} f\partial_i g \partial_i \partial_j^2 h\dx=-\int f (\partial_i\partial_j g)(\partial_i\partial_jh)\dx
-\int (\partial_jf)(\partial_ig)\partial_j^2h\dx.
\end{align*}
Consequently, one has two different expressions for $\Pi_m$:
\begin{align}
P_m&=\int_{\xT^{d}}h^m(\Delta h)^2\dx+m\int_{\xT^{d}}h^{m-1}\la\nabla h\ra^2 \Delta h\dx,\label{Pm1}\\
P_m&=\int_{\xT^{d}}h^m\la\nabla^2 h\ra^2\dx+m\int_{\xT^{d}}h^{m-1}\nabla h\otimes \nabla h:\nabla^2 h\dx.\label{Pm2}
\end{align}
To exploit the fact that there are two different identities for $P_m$, we need to 
figure out the most appropriate linear combination of \e{Pm1} and~\e{Pm2}. 
To do so, we will exploit the following cancellation
\begin{equation}\label{identitysimplei}
\int \Big[f   |\nabla \rho|^2\Delta \rho+2 f \nabla \nabla \rho : \nabla \rho  \otimes \nabla \rho\Big] \dx
=-\int \la \nabla \rho\ra^2\nabla f\cdot\nabla \rho\dx,
\end{equation}
which is proved again by an integration by parts:
$$
\int f (\partial_i\rho)^2\partial_j^2 \rho\dx=-2\int f (\partial_i\rho)(\partial_i\partial_j\rho)\partial_j \rho\dx
-\int (\partial_j f)(\partial_i\rho)^2\partial_j \rho\dx.
$$
This suggests to add  the right-hand side of \e{Pm1} with two times the right-hand side of \e{Pm2}. 
This implies that
$$
3P_m=\int_{\xT^{d}}h^m \Big( 2\la\nabla^2 h\ra^2+(\Delta h)^2\Big)\dx-m(m-1)\int_{\xT^{d}}h^{m-2}\la\nabla h\ra^4\dx.
$$
Now, the functional inequality \e{youpi2} applied with $\mu=m-1$, implies that
$$
\int_{\xT^{d}}h^m \Big( 2\la\nabla^2 h\ra^2+(\Delta h)^2\Big)\dx\ge \frac{(m-1)^2}{3}\int_{\xT^{d}}h^{m-2}\la\nabla h\ra^4\dx.
$$
We thus obtain the wanted lower bound for the dissipation rate:
$$
P_m\ge C_m\int_{\xT^{d}}h^{m-2}\la\nabla h\ra^4\dx\quad\text{with}\quad 
C_m=\frac{1}{3}\Big( \frac{(m-1)^2}{3}-m(m-1)\Big).
$$
Now, it remains only to observe that the above constant $C_m$ is non-negative when 
$-2m^2+m+1\ge 0$, that is for $m$ in $[-1/2,1]$. 
\end{proof}

\section{The Boussinesq equation}\label{S:Boussinesq}
In this section, we begin by studying Lyapunov functionals for the Boussinesq equation
$$
\partial_th-\cnx(h\nabla h)=0.
$$
By a straightforward integration by parts argument, one has the following 
\begin{proposition}\label{convexporoust}
Consider a smooth positive solution to
\begin{equation}\label{boussinesqpasdarcy}
\partial_t h -\cnx (h \nabla h)  = 0.
\end{equation}
For any real number $m$,
\begin{equation}\label{estim1}
\fract\int_{\xT^{d}} h^{m+1}\dx + m(m+1)\int_{\xT^{d}} h^m |\nabla h|^2\dx = 0,
\end{equation}
and 
\begin{equation}\label{estim2}
\fract\int_{\xT^{d}} h\log h \dx+ \int_{\xT^{d}} |\nabla h|^2\dx = 0.
\end{equation}
\end{proposition}

We want to seek strong Lyapunov functionals. 
Recall from Definition~\ref{Defi:1.1} that $I$ is a strong Lyapunov functional if 
it decays in a convex manner, in other words:
$$
\fract I\le 0\quad\text{and}\quad\fractt I\ge 0.
$$
In view of~\e{estim1} and~\e{estim1}, we have to find 
those $m$ for which
$$
\fract\int_{\xT^{d}}h^m\la\nabla h\ra^2\dx\le 0. 
$$
As an example we recall that this property holds for $m=2$. Indeed, as explained by V\'azquez in his monograph (see~\cite[\S$3.2.4$]{Vazquez-PME-book}), 
by multiplying the equation by $\partial_t (h^2)$, one obtains that
$$
\fract\int_{\xT^{d}}h^2\la\nabla h\ra^2\dx\le 0. 
$$
By combining this with~\e{estim1}, we see that the square of the $L^3$-norm is a strong Lyapunov functional. We will 
complement this by considering the square of the 
$L^{m+1}$-norm for $0\le m\le (1+\sqrt{7})/2$. 
The upper bound is quite technical. However, 
for the applications to the classical entropies, the important cases are the lower bound $m=0$ 
together with $m=1$  
(this is because these are the two special results which will be used to prove that the square of the $L^2$-norm 
and the Boltzmann's entropy are strong Lyapunov functionals).

We begin by considering the case $m=1$. In this case, an application 
of the functional inequality given by Proposition~\ref{theo:logSob} will allow us to give a very simple proof of the following
\begin{proposition}
For any smooth positive solution to
\begin{equation}\label{boussinesqpasdarcy2}
\partial_t h -\cnx (h \nabla h)  = 0,
\end{equation}
there holds
\be\label{Boussinesq:L2dtn2}
\mez \fract\int_{\xT^{d}} h\la \nabla h\ra^2\dx+ \int_{\xT^{d}}\Big(\frac16 \la\nabla h\ra^4 +
\mez h^2(\Delta h)^2\Big)\dx\le 0.
\ee
\end{proposition}
\begin{remark}As already mentioned, it follows from~\e{estim1} and \e{Boussinesq:L2dtn2} that
$$
\fractt \int_{\xT^{d}} h^2\dx \ge 0.
$$
This proves that the square of the $L^2$-norm is a strong Lyapunov functional for the Boussinesq 
equation.
\end{remark}
\begin{proof}
The energy equation reads
$$
\frac{1}{2} \fract \int_{\xT^{d}} h^2\dx + \int_{\xT^{d}} h \la \nabla h\ra^2\dx = 0.
$$
Let us study the time derivative of the dissipation rate $\int h|\nabla h|^2\dx$. 
Since
\begin{align*}
\partial_t(h\la \nabla h\ra^2) &= (\partial_t h)\la \nabla h\ra^2 + 2 h \nabla h\cdot \nabla \partial_t h\\
&=  \cnx(h \nabla h)\la \nabla h\ra^2 + 2 h \nabla h \cdot\nabla \cnx(h\nabla h),
\end{align*}
we have
$$
\fract \int_{\xT^{d}} h\la \nabla h\ra^2\dx = \int_{\xT^{d}} \cnx (h \nabla h) \la \nabla h\ra^2\dx 
- 2 \int_{\xT^{d}} (\cnx(h \nabla h))^2\dx.
$$
Remark that 
\begin{align*}
&\cnx (h \nabla h) \la \nabla h\ra^2
= |\nabla h|^4 + h (\Delta h)\la \nabla h\ra^2,\\
&(\cnx(h \nabla h))^2=h^2(\Delta h)^2+\la \nabla h\ra^4+2h(\Delta h)\la \nabla h\ra^2.
\end{align*}   
So we easily verify that
\begin{align*}
\fract \int_{\xT^{d}} h\la \nabla h\ra^2\dx &=
\frac{1}{2} \int_{\xT^{d}} |\nabla h|^4 \dx- \frac{3}{2} \int_{\xT^{d}} ({\rm div}(h\nabla h))^2 \dx \\
&\quad- \frac{1}{2} \int_{\xT^{d}} h^2 (\Delta h)^2\dx.
\end{align*}
Now, we use Proposition~\ref{theo:logSob} applied with $\theta=h^2$ to write that
$$
\int (\cnx(h\nabla h))^2\dx=\uq\int \big(\Delta h^2\big)^2\dx\ge \frac{4}{9}\int \la \nabla h\ra^4\dx.
$$
This completes the proof.
\end{proof}

Let us give now a more general result namely 
\begin{proposition}\label{convexporous}
Consider a smooth positive solution to
\begin{equation}\label{boussinesqpasdarcy3}
\partial_t h -\cnx (h \nabla h)  = 0,
\end{equation}
and a real number $m$ in $[0,(1+\sqrt{7})/2]$. Then 
\begin{equation}\label{estim3}
\fract\int_{\xT^{d}} h^m \la \nabla h\ra^2\dx+ I_m
\le 0,
\end{equation}
with
$$
I_m=\frac{m}{m+1}\int_{\xT^{d}} h^{m+1} |\Delta h|^2 \dx+ C_m \int_{\xT^{d}} h^{m-1} |\nabla h|^4\dx,
$$
where
$$
C_m =\frac{m+2}{m+1}\,  \frac{(m+3)^2}{36}- \frac{m^2- m +2}{4}\ge 0.
$$
\end{proposition}
\begin{remark}
It follows from~\e{estim1} that for all $m$ in $[0,(1+\sqrt{7})/2]$,
$$
\fractt \int_{\xT^{d}} h^{m+1}\dx \ge 0.
$$
Similarly, there holds
$$
\fractt \int_{\xT^{d}}h\log h\dx\ge 0.
$$
\end{remark}
\begin{proof}
Starting from
$$
\partial_t(h^m \la \nabla h\ra^2) = (\partial_t h^m)\la \nabla h\ra^2 + 2 h^m \nabla h\cdot \nabla \partial_t h,
$$
and then using the equation, 
$$
\partial_t h^m - m h^{m-1} {\rm div}(h\nabla h) = 0,
$$
we deduce that
\begin{align*}
\fract \int_{\xT^{d}} h^m\la \nabla h\ra^2\dx
&= \int_{\xT^{d}} m h^{m-1} {\rm div} (h\nabla h) |\nabla h|^2\dx\\
&\quad    + \int_{\xT^{d}} 2 h^m \nabla h \cdot \nabla {\rm div}(h\nabla h)\dx.
\end{align*}
Now we remark that
\begin{align*}
& h^{m-1}\cnx (h \nabla h) \la \nabla h\ra^2
= h^{m-1} |\nabla h|^4 + h^m (\Delta h)\la \nabla h\ra^2 \\
&\cnx(h^m\nabla h) \cnx(h \nabla h)= \Big(\cnx\big(h^{(m+1)/2}\nabla h\big)\Big)^2 - \frac{(m-1)^2}{4}  h^{m-1} |\nabla h|^4.
\end{align*} 
Consequently,
\begin{align*}
\fract \int_{\xT^{d}} h^m\la \nabla h\ra^2\dx 
& =\frac{m^2+1}{2}\int_{\xT^{d}} h^{m-1} |\nabla h|^4 \dx\\
&\quad+ m\int_{\xT^{d}} h^m \Delta h |\nabla h|^2 \dx\\
& \quad- 2 \int_{\xT^{d}}  \big(\cnx\big(h^{(m+1)/2}\nabla h\big)\big)^2\dx.
\end{align*}
By integrating by parts twice, we 
verify that
\begin{align*}
(m+1)\int_{\xT^{d}}h^m\la \nabla h\ra^2\Delta h\dx&=-\int_{\xT^{d}}h^{m+1}(\Delta h)^2\, dx\\
&\quad
+\int_{\xT^{d}}\cnx\big(h^{m+1}\nabla h\big)\Delta h\, dx.
\end{align*}
Then, it follows from the equality
$$\cnx\big(h^{m+1}\nabla h\big) \Delta h
= \Big(\cnx\big(h^{(m+1)/2}\nabla h\big)\Big)^2 - \frac{(m+1)^2}{4}  h^{m-1} |\nabla h|^4,
$$
that
\begin{align*}
\int_{\xT^{d}} h^m\la \nabla h\ra^2\Delta h\, dx
&=-\frac{1}{m+1}\int_{\xT^{d}} h^{m+1}(\Delta h)^2\, dx\\
&\quad+\frac{1}{m+1}\int_{\xT^{d}} \left(\cnx \big(h^{(m+1)/2}\nabla h\big)\right)^2\, dx\\
&\quad-\frac{m+1}{4}\int_{\xT^{d}} h^{m-1}\la \nabla h\ra^4\dx.
\end{align*}
As a result,
\begin{align*}
&\fract \int_{\xT^{d}}  h^m\la \nabla h\ra^2\dx \\
&\qquad\qquad =\frac{m^2-m+2}{4} \int_{\xT^{d}} h^{m-1} |\nabla h|^4\dx \\
&\qquad\qquad \quad - \frac{m}{m+1} \int_{\xT^{d}} h^{m+1} (\Delta h)^2\dx\\
&\qquad\qquad \quad - \frac{m+2}{m+1} \int_{\xT^{d}}
\Big(\cnx\big(h^{(m+1)/2} \nabla h\big)\Big)^2 \dx.
\end{align*}
The inequality \e{youpi2} then implies that
$$
\int_{\xT^{d}} \Big(\cnx\big(h^{(m+1)/2} \nabla h\big)\Big)^2 \dx
\ge \frac{(m+3)^2}{36}\int_{\xT^{d}} h^{m-1} |\nabla h|^4\dx .
$$
Consequently, for any $m\ge 0$,
$$
\fract\int_{\xT^{d}}  h^m \la \nabla h\ra^2\dx+ I_m
\le 0,
$$
with
$$
I_m=\frac{m}{m+1}\int_{\xT^{d}}  h^{m+1} (\Delta h)^2 \dx+ C_m \int_{\xT^{d}} h^{m-1} |\nabla h|^4\dx,
$$
where
$$
C_m = \frac{m+2}{m+1}\cdot\frac{(m+3)^2}{36}-\frac{m^2-m+2}{4}.
$$
By performing elementary computations, one verifies that $C_m\ge 0$ for all $m$ in $[0,(1+\sqrt 7 )/2]$, which completes the proof.
\end{proof}

\appendix 

\section{An application to compressible fluid dynamics}\label{appendix:compressible}

The goal of this appendix is to show that the Sobolev inequality 
given by Proposition~\ref{theo:logSob} has an important application on the
global existence of weak solutions on the compressible 
Navier-Stokes with density dependent viscosities, namely
\begin{equation*}
\left\{
\begin{aligned}
&\partial_t \rho + {\rm div}(\rho u)= 0, \\
&\partial_t (\rho u) + {\rm div}(\rho u\otimes u)
     - 2 {\rm div} (\mu(\rho) D(u)) - \nabla (\lambda(\rho){\rm div} u ) + \nabla p(\rho)= 0,
\end{aligned}
\right.
\end{equation*}
with $D(u) = (\nabla u + {}^t\nabla u)/2$, $p(s)=a s^\gamma$ with $\gamma>1$ and the initial boundary conditions
$$
\rho\vert_{t=0} = \rho_0, \qquad 
\rho u\vert_{t=0} = m_0.
$$
Recently Bresch, Vasseur and Yu~\cite{BrVAYu19} 
obtained the first result with a large class of given shear and bulk viscosities 
respectively $s\mapsto \mu(s)$ and $s\mapsto\lambda(s)$ in a periodic domain $\Omega = {\mathbb T}^3$. 
More precisely, if we assume the shear and bulk viscosities as 
\begin{equation}
\mu(\rho)= \rho^\alpha,
\qquad
\lambda(\rho) = 2(\alpha-1) \rho^\alpha,
\end{equation}
then the authors obtained
the existence of solutions under the assumption that
$$
\frac{2}{3} < \alpha < 4.
$$
The lower bound is a constraint coming naturally from a necessary coercivity property. 
The upper-bound is a mathematical constraint due to Lemma $2.1$ in~\cite{BrVAYu19}, which reads as follows:  
There exists $C>0$ independent on $\alpha$ and $\varepsilon >0$ as small as we want such that
\be
\begin{aligned} \nonumber 
+\infty > \frac{C}{\varepsilon}\int \rho^\alpha |\nabla\nabla \rho^{\alpha-1}|^2
&  \ge \frac{4}{(3\alpha-2)^2} 
 \int |\nabla^2 \rho^{(3\alpha-2)/2}|^2 \\
& + \left(\frac{1}{\alpha}- \frac{1}{4} - \varepsilon\right) \frac{4^4}{(3\alpha-2)^4}
    \int |\nabla \rho^{(3\alpha-2)/4}|^4.
\end{aligned}
\ee 
The constraint $\alpha<4$ allows to have two
positive terms in the righ-hand side and therefore
some appropriate controls on $\rho$ namely
$$
\nabla^2 \rho^{(3\alpha-2)/2}\in L^2((0,T)\times {\mathbb T}^3)\quad\text{and}\quad
\nabla\rho^{(3\alpha-2)/4} \in  L^4((0,T)\times {\mathbb T}^3).
$$
Proposition \ref{theo:logSob} allows to compare the first
and the second quantity and therefore to relax the
constraint $\alpha <4.$ More precisely, using such estimate, 
it suffices to check that
$$
\frac{1}{9} + 
\Bigl(\frac{1}{\alpha}-\frac{1}{4}\Bigr)
\frac{4}{(3\alpha-2)^2} >0,
$$
to get a positive quantity on the right-hand side 
controlling the $H^2$ derivatives. 
We can check that it true for all $\alpha$ such 
that $2/3 < \alpha <+\infty$. This implies 
that the result by Bresch--Vasseur--Yu still holds 
for any $\mu$ and $\lambda$ such that
$$
\mu(\rho) = \rho^\alpha, \qquad
    \lambda(\rho) = 2 (\alpha-1) \rho^\alpha
$$
with $2/3 < \alpha <+\infty$.

\section{Lyapunov functionals for the mean-curvature equation}\label{Appendix:MCF}
\begin{proposition}\label{Prop:C1nabla}
If $h$ is a smooth solution to the mean-curvature equation 
$$
\partial_th+\sqrt{1+|\nabla h|^2}\kappa=0\quad\text{with}\quad
\kappa=-\cnx\left(\frac{\nabla h}{\sqrt{1+|\nabla h|^2}}\right),
$$
then
\be\label{nC10}
\fract \int_{\xT^d}\la \nabla h\ra^2\dx\le 0.
\ee
\end{proposition}
\begin{proof}
By multiplying the equation by $-\Delta h$ and integrating by parts, we find that
$$
\fract \int_{\xT^d}\la \nabla h\ra^2\dx -\int_{\xT^d}\sqrt{1+|\nabla h|^2}\kappa\Delta h\dx=0.
$$
Using the Leibniz rule, one has
$$
-\sqrt{1+|\nabla h|^2}\kappa\Delta h
=(\Delta h)^2-\frac{\nabla h\cdot (\nabla h\cdot \nabla\nabla h)\Delta h}{1+|\nabla h|^2}.
$$
It follows from the Cauchy-Schwarz inequality that
$$
\la \frac{\nabla h\cdot (\nabla h\cdot \nabla\nabla h)\Delta h}{1+|\nabla h|^2}
\ra
\le \la \nabla^2 h\ra\la \Delta h\ra.
$$
Consequently,
$$
-\int_{\xT^d}\sqrt{1+|\nabla h|^2}\kappa\Delta h\dx
\ge \int_{\xT^d}\big((\Delta h)^2-\la \nabla^2 h\ra\la \Delta h\ra\big)\dx.
$$
Now we claim that the above term is non-negative, which in turn will imply the 
wanted result~\e{nC10}. To see this, we first use the Cauchy-Schwarz inequality to bound this term from below by
$$
\int_{\xT^d}(\Delta h)^2\dx
-\left(\int_{\xT^d}(\Delta h)^2\dx\right)^\mez\left(
\int_{\xT^d}\la \nabla^2 h\ra^2\dx\right)^\mez,
$$
and then apply the classical identity (see~\e{Deltanablanabla})
$$
\int_{\xT^d}(\Delta h)^2\dx=\int_{\xT^d}\la \nabla h\ra^2\dx,
$$
which can be verified by integrating by parts twice.
\end{proof}
\begin{proposition}\label{Prop:C1}
If $h$ is a smooth solution to the mean-curvature equation in space dimension $d=1$:
$$
\partial_t h+\sqrt{1+(\partial_x h)^2}\kappa=0\quad\text{with}\quad
\kappa=-\partial_x\left(\frac{\partial_x h}{\sqrt{1+(\partial_x h)^2}}\right),
$$
then the following quantities are Lyapunov functionals:
$$
\int_\xT h^2\dx,\quad 
\int_\xT (\partial_t h)^2\dx,\quad \int_\xT (1+(\partial_xh)^2)\kappa^2\dx. 
$$
In addition, $\int_\xT h^2\dx$ is a strong Lyapunov functional.
\end{proposition}
\begin{proof}
If the space dimension $d$ is equal to~$1$, we have
$$
\sqrt{1+(\partial_x h)^2}\kappa=-\frac{\partial_{xx} h}{1+(\partial_x h)^2}.
$$
Consequently, the one-dimensional version of the mean-curvature equation reads
$$
\partial_t h-\frac{\partial_{xx}h}{1+(\partial_xh)^2}=0.
$$
We may further simplify the mean curvature equation by noticing that
\be\label{MCFarctan}
\partial_t h+\sqrt{1+(\partial_xh)^2}\kappa=\partial_th-\frac{\partial_{xx}h}{1+(\partial_xh)^2}=
\partial_th-\partial_x \arctan (\partial_xh).
\ee
This immediately implies that the square of the $L^2$-norm is a Lyapunov functional:
\be\label{n71}
\mez\fract \int_{\xT}h^2\dx=-\int_\xT (\partial_xh)\arctan (\partial_xh)\dx \le 0,
\ee
since $u \arctan u\ge 0$ for all $u\in \xR$. 

It also follows from the previous $\arctan$-formulation that the unknown $\dot{h}=\partial_t h$ is solution to
$$
\partial_t\dot{h}-\partial_x\left( \frac{\partial_x 
\dot{h}}{1+(\partial_xh)^2}\right)=0.
$$
Multiplying the previous equation by $\dot{h}$ and integrating by parts in $x$, we infer that
$$
\fract \int_\xT(\partial_th)^2\dx\le 0. 
$$
By using the equation for $h$, this is equivalent to
$$
\fract \int_\xT (1+(\partial_xh)^2)
\kappa^2 \dx\le 0. 
$$

Now observe that 
$$
\partial_t \big((\partial_x h) \arctan (\partial_x h)\big)
     = \partial_t\partial_x h \Bigl(\arctan (\partial_x h)
          + \frac{\partial_x h}{1+(\partial_x h)^2}\Bigr).
$$
On the other hand, using the equation \e{MCFarctan}, we have
$$
\partial_t\partial_x h=\partial_x\Big(\frac{\partial_{xx}h}{1+(\partial_xh)^2}\Big).
$$
Therefore, integrating by parts, we conclude that
\begin{align*}
&\fract \int_\xT (\partial_xh)\arctan (\partial_xh)\dx \\
&\qquad\qquad=-\int_\xT \frac{\partial_{xx}h}{1+(\partial_xh)^2} 
\partial_x \left(\arctan (\partial_x h)
          + \frac{\partial_x h}{1+(\partial_x h)^2}\right)\dx \\
          &\qquad\qquad=-\int_\xT \frac{\partial_{xx}h}{1+(\partial_xh)^2} \cdot 
\frac{2\partial_{xx} h}{(1+(\partial_x h)^2)^2}\dx.
\end{align*}
This proves that
$$
\fract \int_\xT (\partial_xh)\arctan (\partial_xh)\dx=-2\int_\xT\kappa^2\dx \le 0. 
$$
So, in view of \e{n71}, we conclude that
$$
\fractt \int_{\xT}h^2\dx\ge 0.
$$
We thus have proved that
$$
\fract \int_{\xT}h^2\dx\le 0\quad\text{and}\quad \fractt \int_{\xT}h^2\dx\ge 0.
$$
By definition, this means that $\int_{\xT}h^2\dx$ is a strong 
Lyapunov functional for the mean-curvature equation.
\end{proof}

The next proposition gives a somewhat surprising property of 
the Boussinesq equation, which is directly inspired by the $\arctan$-formulation used above for the 
mean-curvature equation.

\begin{proposition}\label{prop:C2Boussinesq}
Consider the Boussinesq equation in space dimension $1$:
$$
\partial_th-\partial_x(h\partial_x h)=0.
$$
Then
$$
\fract \int_{\xT} (\partial_x h) \arctan (\partial_x h)\dx \le 0.
$$
\end{proposition}
\begin{proof}
As already seen in the previous proof, 
$$
\partial_t \big((\partial_x h) \arctan (\partial_x h)\big)
=\partial_t\partial_x h \left(\arctan (\partial_x h)+\frac{\partial_x h}{1+(\partial_x h)^2}\right).
$$
Using the equation
$$
\partial_t\partial_x h = \partial_x^2(h\partial_x h),
$$ 
and then integrating by parts, we get
$$
\fract \int_{\xT} (\partial_x h) \arctan (\partial_x h)\dx
  =  - 2\int_{\xT} \partial_x(h\partial_x h)
        \frac{\partial_x^2 h}{(1+(\partial_x h)^2)^2}\dx
  = I,
$$
where $I$ reads
$$
I=-  \int_{\xT}\frac{2h(\partial_x^2 h)^2}{(1+(\partial_x h)^2)^2}\dx
   - \int_{\xT}\frac{2(\partial_x h)^2 \partial_x^2 h}{(1+(\partial_x h)^2)^2}\dx.
$$
Note that the second term vanishes since this is the integral of an exact derivative. 
So,
$$
\fract \int_{\xT} (\partial_x h) \arctan (\partial_x h)\dx
+\int_{\xT}\frac{2h(\partial_x^2 h)^2}{(1+(\partial_x h)^2)^2}\dx
=0,
$$
which implies the wanted conclusion.
\end{proof}

\section{A Rellich type estimate}\label{A:Rellich} 

This appendix gives a proof of the inequality~\e{d10}. 
\begin{lemma}
For any smooth functions $h$ and $\zeta$ in $C^\infty(\xT^d)$, there holds
\be\label{d10-bisb}
\int_{\xT^d} (G(h)\zeta)^2\dx \le   
\int_{\xT^d} (1+|\nabla h|^2)|\nabla \zeta-\mathcal{B} \nabla h|^2 \dx,
\ee
where
\be\label{d11-bisb}
\mathcal{B}=\frac{G(h)\zeta+\nabla \zeta \cdot \nabla h}{1+|\nabla h|^2}.
\ee
\end{lemma}
\begin{remark}
$i)$ This inequality extends to functions which are not smooth. 

$ii)$ This generalizes an estimate proved in~\cite{A-stab-AnnalsPDE} when $d=1$, 
for the Dirichlet-to-Neumann operator associated 
to a domain with finite depth. When $d=1$, the main difference is that this is an identity (and not only an inequality). 
This comes from the fact that, in the proof below, to derive \e{esti:final7} we use the inequality 
$(\nabla h\cdot \mathcal{V})^2\le |\nabla h|^2 \cdot |\mathcal{V}|^2$, which is clearly an equality when $d=1$.
\end{remark}
\begin{proof} 
We follow the analysis in~\cite{A-stab-AnnalsPDE}. Set
$$
\Omega=\{(x,y)\in\xT^{d}\times\xR\,;y<h(x)\},
$$
and denote by $\phi$ the harmonic function defined by 
\begin{equation}\label{m1}
\left\{
\begin{aligned}
&\Delta_{x,y}\phi=0\quad\text{in }\Omega=\{(x,y)\in \xT\times \xR \,;\, y<h(x)\},\\
&\phi(x,h(x)) = \zeta(x).
\end{aligned}
\right.
\end{equation} 
As recalled in Lemma~\ref{Lemma:decayinfty}, this 
is a classical elliptic boundary problem, which admits a unique 
smooth solution. Moreover, it satisfies
\be\label{decaytozero-appendix}
\lim_{y\to-\infty}\sup_{x\in\xT^{d}}\la \nabla_{x,y}\phi(x,y)\ra=0.
\ee
Introduce the notations
$$
\mathcal{V}=(\nabla\phi)_{\arrowvert y=h}, \qquad \mathcal{B}=(\partial_y\phi)_{\arrowvert y=h}.
$$
(We parenthetically recall that $\nabla$ denotes the gradient with respect to 
the horizontal variables 
$x=(x_1,\ldots,x_d)$ only.) 
It follows from the chain rule that 
$$
\mathcal{V}=\nabla \zeta-\mathcal{B}\nabla h,
$$
while $\mathcal{B}$ is given by \e{d11-bisb}. 
On the other hand, by definition of the Dirichlet-to-Neumann operator, one has
the identity
$$
G(h)\zeta=\big(\partial_y \phi-\nabla h\cdot \nabla \phi\big)_{\arrowvert y=h},
$$
so 
$$
G(h)\zeta=\mathcal{B}-\nabla h\cdot \mathcal{V}.
$$
Squaring this identity yields
$$
(G(h)\zeta)^2
=\mathcal{B}^2-2 \mathcal{B}\nabla h \cdot \mathcal{V} +(\nabla h\cdot \mathcal{V})^2.
$$
Since $(\nabla h\cdot \mathcal{V})^2\le |\nabla h|^2 \cdot |\mathcal{V}|^2$, this implies the inequality:
\be\label{esti:final7}
(G(h)\zeta)^2\le \mathcal{B}^2-\la\mathcal{V}\ra^2-2\mathcal{B}\nabla h\cdot \mathcal{V} +(1+|\nabla h|^2)\mathcal{V}^2.
\ee
Integrating this gives
$$
\int_{\xT^d} (G(h)\zeta)^2\dx \le
\int_{\xT^d} (1+|\nabla h|^2)\la\mathcal{V}\ra^2 \dx+R,
$$
where
$$
R=\int_{\xT^d}\Big( \mathcal{B}^2-\la\mathcal{V}\ra^2-2\mathcal{B}\nabla h\cdot \mathcal{V}\Big)\dx.
$$
Since $\la\mathcal{V}\ra=|\nabla \zeta-\mathcal{B} \nabla h|$, 
we immediately see that, to obtain the wanted estimate~\e{d10-bisb}, it is sufficient to prove 
that $R=0$. To do so, we begin by noticing that $R$ is the flux associated to a vector field. Indeed, 
$$
R=\int_{\partial\Omega} X\cdot n\diff\Hm
$$
where 
$X\colon \Omega\rightarrow \xR^{d+1}$ is given by
$$
X=(-(\partial_y\phi)\nabla \phi;|\nabla \phi|^2-(\partial_y\phi)^2).
$$ 
Then the key observation is that this vector field satisfies
$\cn_{x,y} X=0$ since
$$
\partial_y \big( (\partial_y\phi)^2-|\nabla\phi|^2\big)
+2\cnx \big((\partial_y\phi)\nabla\phi\big)=
2(\partial_y\phi) \Delta_{x,y}\phi=0,
$$
as can be verified by an elementary computation. 
Now, we see that the cancellation $R=0$ comes from the Stokes' theorem. 
To rigorously justify this point, we 
truncate $\Omega$ in order to work in a smooth bounded domain. Given a parameter $\beta>0$, set
$$
\Omega_\beta=\{(x,y)\in\xT^{d}\times\xR\,;-\beta<y<h(x)\}.
$$
An application of the divergence theorem in $\Omega_\beta$ gives that
$$
0=\iint_{\Omega_\beta} \cn_{x,y}X\dydx=R+\int_{\{y=-\beta\}}X\cdot (-e_y)\dx,
$$
where $e_y$ is the vector $(0,\ldots,0,1)$ in $\xR^{d+1}$.
Sending $\beta$ to $+\infty$ and remembering that $X$ converges to $0$ uniformly when $y$ goes to $-\infty$ (see~\e{decaytozero-appendix}), we obtain the expected result $R=0$. 
This completes the proof.
\end{proof}

\section{Darcy's law}\label{appendix:HS}
In this appendix, we recall the derivation of the Hele-Shaw and Mullins-Sekerka equations. 
These equations dictate the dynamics of the free surface of an incompressible 
fluid evolving according to 
Darcy's law. Consider a time-dependent fluid domain $\Omega$ 
of the form:
$$
\Omega(t)=\{ (x,y) \in \xT^{d}\times \xR\,;\, y < h(t,x)\}.
$$
The Darcy's law stipulates that the velocity 
$v\colon \Omega\rightarrow \xR^{d+1}$ and the pressure $P\colon\Omega\rightarrow \xR$ 
satisfy the following equations:
$$
\cnxy v=0\quad\text{ and }\quad  v=-\nabla_{x,y} (P+gy) \quad \text{in }\Omega,
$$
where $g>0$ is the acceleration of gravity. 
In addition, one assumes that
$$
\lim_{y\to-\infty}v=0
$$
and that, on the free surface $\partial\Omega$, 
the normal component of $v$ 
coincides with the normal component of the velocity of free surface, which 
implies that
$$
\partial_t h=\sqrt{1+|\partialx h|^2} \, v\cdot n\quad \text{on}\quad y=h,
$$
where $\nabla=\nabla_x$ and $n$ is the outward unit normal to $\partial\Omega$, given by
$$
n=\frac{1}{\sqrt{1+|\nabla h|^2}} \begin{pmatrix} -\nabla h \\ 1 \end{pmatrix}.
$$
The final equation states that
the restriction of the pressure to the free surface is proportional to the mean curvature:
$$
P=\mu \kappa \quad \text{on}\quad\partial\Omega,
$$
where the parameter $\mu$ belongs to  $[0,1]$ and 
$\kappa$ is given by~\e{defi:kappa}.

Now we notice that $\Delta_{x,y}(P+gy)=\cnxy v=0$ so $P+gy$ is 
the harmonic extension of $gh+\mu\kappa$. 
It 
follows that the Hele-Shaw problem is equivalent to
\begin{equation*}
\partial_{t}h+G(h)(gh+\mu \kappa)=0.
\end{equation*}


\vspace{1cm}

\begin{flushleft}
\textbf{Thomas Alazard}\\
UniversitŽ Paris-Saclay, ENS Paris-Saclay, CNRS,\\
Centre Borelli UMR9010, avenue des Sciences, \\
F-91190 Gif-sur-Yvette\\

\vspace{1cm}

\textbf{Didier Bresch}\\
LAMA CNRS UMR5127, Univ. Savoie Mont-Blanc, \\
Batiment le Chablais, \\
F-73376 Le Bourget du Lac, France.

\end{flushleft}

\end{document}